\documentclass{amsart}
\usepackage{amsmath,amsthm,amscd,amssymb, color}
\usepackage{latexsym}
\usepackage{fullpage}
\usepackage{bm}
\usepackage{graphicx}
\usepackage[all]{xy}

\newcommand{\F}{\mathbb F}
\newcommand{\Z}{\mathbb Z}
\newcommand{\FF}{{\mathcal {F}}}
\newcommand{\FFd}{{\mathcal {F}_d}}
\newcommand{\PP}{\mathbb P}
\newcommand{\thet}{\theta_j(f, \psi)}

\DeclareMathOperator{\tr}{tr}

\DeclareMathOperator{\SL}{SL}
\DeclareMathOperator{\full}{full}
\DeclareMathOperator{\ord}{ord}

\newtheorem{lem}{Lemma}[section]
\newtheorem{prop}[lem]{Proposition}
\newtheorem{thm}[lem]{Theorem}
\newtheorem{defn}[lem]{Definition}
\newtheorem{cor}[lem]{Corollary}

\theoremstyle{definition}

\newtheorem{rem}[lem]{Remark}

\date\today
\title{Statistics for ordinary Artin-Schreier covers and other $p$-rank strata}
\author{Alina Bucur, Chantal David, Brooke Feigon, Matilde Lal\'{i}n}
\address{Alina Bucur: Department of Mathematics, University of California at San Diego,
9500 Gilman Drive $\#$0112, La Jolla, CA 92093, USA} \email{alina@math.ucsd.edu}
\address{Chantal David: Department of Mathematics and Statistics,
Concordia University,
1455 de Maisonneuve West,
Montreal, QC H3G 1M8, Canada} \email{cdavid@mathstat.concordia.ca}
 \address{Brooke Feigon: Department of Mathematics,
The City College of New York,
CUNY,
NAC 8/133,
New York, NY 10031, USA} \email{bfeigon@ccny.cuny.edu }
 \address{Matilde Lal\'in: D\'epartement de math\'ematiques et de statistique,
                                    Universit\'e de Montr\'eal.
                                    CP 6128, succ. Centre-ville.
                                     Montreal, QC H3C 3J7, Canada} \email{mlalin@dms.umontreal.ca}
\begin{document}

\begin{abstract} We study the distribution of the number of points and of the zeroes of the zeta function in different $p$-rank strata of Artin-Schreier covers over $\F_q$ when $q$ is fixed and
the genus goes to infinity. The $p$-rank strata considered include the ordinary family, the whole family, and the family of curves with $p$-rank equal to $p-1.$ While the  zeta zeroes always approach the standard Gaussian distribution, the number of points
over $\F_q$ has a distribution that varies with the specific family.
\end{abstract}



\maketitle
\tableofcontents

\section{Introduction}

Besides their central place in number theory, algebraic curves over finite fields also play a pivotal role in applications via such fields as cryptography and error-correcting codes. In both theory and applications, a key property of an algebraic curve over a finite field is its {\em zeta function}, which determines and is determined by the number of points on the curve over the finite extensions of the base field. These zeta functions exhibit a strong analogy with other zeta functions occurring in number theory, such as the Riemann zeta function, with the added benefit that the analogue of the Riemann hypothesis is known by results of Weil.

In addition to studying curves individually, it is also profitable to study curves in families and ask aggregate questions over families. Historically, this generally involved varying the finite field, as in the work of Deligne. More recently, a series of results have emerged in which the finite field is fixed and other geometric parameters are
allowed to vary. Examples include the work Kurlberg and Rudnick \cite{kr} that studies the distribution of the number of points on hyperelliptic curves as the genus grows. Similar statistics for the number of points have been computed for cyclic $\ell$-covers of the projective line \cite{bdfl1, bdfl2, x2}, plane curves \cite{bdfl3}, complete intersections in projective spaces \cite{bk}, general trigonal curves \cite{wood}, superelliptic curves \cite{cwz},  curves on Hirzebruch surfaces \cite{ew}, and a subfamily of Artin-Schreier covers \cite{entin}.

A finer statistic for these curves is the distribution of the zeroes of the zeta function. (Note that the distribution of the points can be easily deduced from the distribution of the zeroes.) The problem of the distribution of the zeroes in the global and mesoscopic regimes was considered  by Faifman and Rudnick \cite{fr} for hyperelliptic curves while \cite{x1}, \cite{x3}, and \cite{BDFLS} treated the cases of cyclic $\ell$-covers, abelian covers of algebraic curves, and Artin-Schreier covers respectively. On the other hand, Entin \cite{entin} used the distributions of the number of points of a subfamily of Artin-Schreier covers to obtain some partial results towards the pair correlation problem for the zeroes.

Artin-Schreier curves represent a special family because they cannot be uniformly obtained by base-changing a scheme defined over $\Z.$  This is intimately related to the fact that their zeta function has an expression in terms of {\it additive} characters of $\F_p$, and not in terms of multiplicative characters as  is the case for the family of hyperelliptic curves and cyclic $\ell$-covers. On the other hand, the factor corresponding to a fixed additive character has a nice description as an exponential sum \eqref{Euler-product}, which allows one to do a fair number of concrete computations. For instance, they can sometimes be  used to show that the Weil bound on the number of points is sharp (especially in the supersingular case \cite{garcia, gv}).

The $p$-rank induces a stratification on the moduli space of Artin-Schreier covers of genus $\mathfrak g.$ We would like to remark that this stratification is not specific to  Artin-Schreier covers. Perhaps the best known example is the case of elliptic curves. The moduli space of elliptic curves only has two $p$-strata -- $p$-rank $1$ (ordinary) and  $p$-rank $0$ (supersingular) --  and these two classes of elliptic curves behave fundamentally differently in many aspects. The ordinary stratum is Zariski dense in the moduli space, but there are only finitely many  supersingular $\bar \F_q$-points in the moduli space of elliptic curves.

In the case of the Artin-Schreier covers, the picture is more complicated, as there are many intermediate strata besides the minimal $p$-rank and the maximal $p$-rank stratum. But it is still the case that the $p$-rank $0$ stratum, when non-empty, is the smallest stratum in the moduli space $\mathcal{AS}_{\mathfrak g}$ of Artin-Schreier covers of genus $\mathfrak g.$ However, the $p$-rank $0$ stratum appears if and only if  $2\mathfrak{g}/(p-1) \not \equiv -1 \pmod p.$ Moreover, the supersingular locus is usually strictly contained in this stratum and it is not easy to locate the supersingular covers among those with $p$-rank $0.$ (See \cite{zhu}.) On the other hand, the maximal $p$-rank stratum is irreducible in $\mathcal{AS}_{\mathfrak g},$ and in some sense, it is still the biggest stratum. As it is noted in \cite[Example 2.9]{pz}, in the case of $p\geq 3$ that we are interested in, the ordinary locus is nonempty whenever $2\mathfrak{g}/(p-1)$ is even. Otherwise, we can still talk about the stratum of maximal $p$-rank; but that maximal rank will be strictly smaller than the genus (namely equal to $\mathfrak g - \frac{p-1}{2}$), and there is no ordinary locus.

Fix a finite field $\F_q$ of odd characteristic $p.$ An Artin-Schreier cover is an Artin-Schreier curve for which we fix an automorphism of order $p$ and an isomorphism between the quotient and $\mathbb P^1.$ Concretely, an $\F_q$-point of the moduli space of Artin-Schreier covers of genus $\mathfrak g$ consists, up to $\F_q$-isomorphism, of a curve of genus $\mathfrak g$ with affine model
\[C_f: y^p - y = f(x),\] where $f(x)\in \F_q(x)$ is a rational function, together with the automorphism $y \mapsto y+1.$

Let $p_1, \dots, p_{r+1}$ be the set of poles of $f(x)$ and let $d_j$ be the order of the pole $p_j$. By Artin--Schreier theory, we can assume that $p \nmid d_j$. Then the genus of $C_f$ is given by
\begin{equation}\label{genusformula}
\mathfrak{g}(C_f)=\frac{p-1}{2}\left(-2 + \sum_{j=1}^{r+1} (d_j+1)\right)=\frac{p-1}{2}\left(r-1+\sum_{j=1}^{r+1} d_j\right).
\end{equation}
(See \cite[Lemma 2.6]{pz}.) The $p$-rank is the integer $s$ such that the cardinality of $\mathrm{Jac}(C_f)[p](\overline{\F_q})$ is $p^s$; by the Deuring-Shafarevich formula, we have $s=r(p-1)$
where  $r+1$ is the number of poles of $f(x)$. We will write $\mathcal{AS}_{\mathfrak g, s}$ for the stratum with $p$-rank equal to $s$ of the moduli space $\mathcal{AS}_{\mathfrak g}.$ For example, $s=0$ corresponds to one pole, which can always be moved to infinity. This is the case where $f(x)$ is a polynomial
that was considered in \cite{entin, BDFLS}. However, this case only corresponds to a piece, namely $\mathcal{AS}_{\mathfrak{g},0}$, of the whole moduli space $\mathcal{AS}_\mathfrak{g}$ of Artin-Schreier covers of genus $\mathfrak{g}$. The next case is $s=p-1$, which includes the case when $f(x)$ is a Laurent polynomial, but this is not the only way one get this $p$-rank, as we explain in Section \ref{sec:prescribed}.   For details on the moduli space of Artin-Schreier curves and the
$p$-rank stratification, we refer the reader to \cite{pz}.

The main object of this paper is the study of the distribution of the number of points and zeta zeroes for the  ordinary locus $\mathcal{AS}_{\mathfrak{g},\mathfrak{g}}$  which only appears when $2\mathfrak{g}/(p-1)$ is even. In addition, we treat the cases of $\mathcal{AS}_{\mathfrak{g}, p-1}$ of covers with $p$-rank equal to $p-1$  and  the whole family $\mathcal{AS}_{\mathfrak{g}}$.
More precisely, we have the following results.
\begin{thm} \label{thmnumberofpoints}
\begin{enumerate}\item Assume that $2\mathfrak{g}/(p-1)$ is even. The average number of $\F_{q^k}$-points on an ordinary Artin-Schreier cover in $\mathcal{AS}_{\mathfrak{g},\mathfrak{g}}(\F_q)$ is
\[\begin{cases} q^{k} +1 + O\left(q^{(-1/2+\varepsilon)\left(1+ \mathfrak g/(p-1)\right)+2k}\right)& p \nmid k, \\\\
 q^{k} +1 +\frac{p-1}{1+q^{-1}-q^{-2}} +\displaystyle\sum_{u\mid \frac{k}{p}} \frac{p-1}{1+q^{-u}-q^{-2u}} \sum_{e \mid u} \mu(e) q^{u/e}+  O\left(q^{(-1/2+\varepsilon)\left(1+ \mathfrak g/(p-1)\right)+2k}\right)& p \mid k.\end{cases}\]
\item The average number of $\F_{q^k}$-points on an Artin-Schreier cover in $\mathcal{AS}_{\mathfrak{g}}(\F_q)$ whose ramification divisor is supported at $r+1$ points and has degree $d$
is
\[\begin{cases} q^{k} +1 +O\left(q^{(\varepsilon-1)d+2k}\right)& p \nmid k, \\ \\
q^k+1+(p-1)q^{k/p}+\frac{p-1}{1+q^{-1}}-(p-1)\displaystyle \sum_{u|\frac{k}{p}}\frac{1}{1+q^u}\sum_{e\mid u}\mu(e)q^{u/e}+O\left(q^{(\varepsilon-1)d+2k}\right)& p \mid k. \end{cases}\]
\item The average number of $\F_{q^k}$-points on an Artin-Schreier cover in $\mathcal{AS}_{\mathfrak{g},p-1} (\F_q)$  is
\[\begin{cases} q^{k} +1 & p \nmid k, \\\\ q^{k} +1 +(p-1)(q^{k/p}-1) & p \mid k,\, k \mbox{ even},\\\\
q^{k} +1 +(p-1)q^{k/p} & p \mid k,\, k \mbox{ odd}. \end{cases}\]
\end{enumerate}

\end{thm}

By Weil's conjectures, the zeta function of $C_f$,
\[Z_{C_f} (u) = \exp \left( \sum_{k=1}^{\infty} N_k(C_f) \frac{u^k}{k} \right), \]
where $N_k(C_f)$ is the number of points on $C_f$ defined over $\F_{q^k}$,
can be written as
\[Z_{C_f} (u) =
\frac{P_{C_f}(u)}{(1-u)(1-qu)},\] where $P_{C_f}(u)$ is a polynomial of degree $2\mathfrak{g}=(p-1)(\Delta-1)$ with $\Delta=r+\sum_{j=1}^{r+1} d_j$. Using Lemma \ref{expectednumberfiber} and the additive characters
 of $\F_p$ to write a formula for $N_k(C_f)$, it follows easily that
\begin{equation}
P_{C_f}(u) = \prod_{\psi} L(u, f, \psi),\end{equation}
where the product is taken over the {\it non-trivial} additive characters $\psi$ of $\mathbb F_p$, and $L(u, f, \psi)$ are certain $L$-functions (given later by \eqref{Euler-product}). Understanding the distribution of the zeroes of $Z_{C_f} (u)$ amounts to understanding the
distribution of the zeroes of each of the $L(u, f, \psi)$ as $f$ runs in the relevant family of rational functions and the genus goes to infinity.

If we write
\[L(u,f,\psi)=\prod_{j=1}^{\Delta-1} (1-\alpha_j(f,\psi)u),\]
we have that $\alpha_j(f, \psi) = \sqrt{q} e^{2 \pi i \theta_j(f, \psi)}$ and $\theta_j(f, \psi) \in [-1/2, 1/2)$.
We study the statistics of the set of angles $\{\theta_j(f, \psi)\}$ as $f$ varies in the family. For an interval $\mathcal{I} \subset[-1/2,1/2),$ let
\[N_\mathcal{I}(f, \psi) := \#\{1\leq j \leq \Delta-1:\,\theta_j(f, \psi) \in \mathcal{I}\},\]
and
\[
N_\mathcal{I}(C_f):=\sum_{j=1}^{p-1}N_\mathcal{I}(f,\psi^j).
\]
We show that the number of zeroes with angle in a prescribed non-trivial subinterval $\mathcal I$
is asymptotic to $2\mathfrak{g}|\mathcal I|$, has variance asymptotic to $\frac{2(p-1)}{\pi^2} \log(\mathfrak{g}|\mathcal I|)$, and properly normalized has a Gaussian distribution.

\begin{thm}\label{zeroesthm}  Fix a finite field $\F_q$ of characteristic $p$. Let $\mathcal{AS}$ denote the family of Artin-Schreier covers, ordinary Artin-Schreier covers, or the $p$-rank $p-1$ Artin-Schreier covers.  Then for any real numbers $a<b$ and  $0<|\mathcal{I}|<1$ either  fixed or
$|\mathcal{I}|\rightarrow 0$ while $\mathfrak{g}|\mathcal{I}|\rightarrow \infty$,
\[\lim_{\mathfrak g \rightarrow \infty} \mathrm{Prob}_{\mathcal{AS}(\mathbb{F}_q)}\left(a < \frac{N_\mathcal{I}(C_f)-2\mathfrak{g}|\mathcal{I}|}{\sqrt{\frac{2(p-1)}{\pi^2}\log\left(\mathfrak{g}|\mathcal{I}|\right)}} < b\right)=\frac{1}{\sqrt{2\pi}} \int_a^b e^{-x^2/2} dx.\]
\end{thm}
This result is analogous to what was obtained in \cite{BDFLS} for $p$-rank $0$ Artin-Schreier covers and is compatible with the philosophy of Katz and Sarnak \cite{KS}. In fact, Katz \cite{Katz1} shows that the monodromy of the $L$-functions defined in \eqref{Euler-product} is given by $\SL\left(2\mathfrak g/(p-1)\right)$ when the dimension of the moduli space is big enough. Since the dimension grows with the genus, this occurs when $\mathfrak g$ is big enough. In particular,  \cite{ds} implies that the limiting distribution as $\mathfrak g \to \infty$ is Gaussian.  

\begin{rem}
 A similar result can be proved for $N_\mathcal{I}(f, \psi)$ with asymptotic mean and variance $(\Delta-1)|\mathcal{I}|$ and $\frac{1}{\pi^2}\log \mathfrak{g}|\mathcal{I}|$ respectively with the additional restriction that the interval $\mathcal{I}$ is symmetric. In fact, under this condition, the $N_\mathcal{I}(f, \psi^j)$ for $j=1,\dots, (p-1)/2$ approach independently jointly normal distributions.
\end{rem}

\section{Basic Artin-Schreier theory}

Fix an odd prime $p$ and let $\mathbb F_q$ be a finite field of characteristic $p$ with $q$ elements.
We consider, up to $\F_q$-isomorphism, pairs of curves with affine model \[C_f: y^p - y = f(x)\] with $f(x)$ a rational function  together with the automorphism $y \mapsto y+1.$

For each integer $n\geq 1$, denote by $\tr_n: \mathbb F_{q^n} \to \mathbb F_p$ the absolute trace map (not the trace to $\F_q$).

\begin{lem} \label{expectednumberfiber}
For each $\alpha \in \PP^1(\F_{q^n})$, the number of points on the curve $C_f: y^p - y = f(x)$
in the fiber above $\alpha$ which are defined over $\F_{q^n}$ is given by  
$$\begin{cases} 1  &  \textrm{if } {f(\alpha)} =  \infty, \\
&\\
 p  &  \textrm{if } f(\alpha)  \in \mathbb F_{q^n} \textrm{ with }\tr_n f(\alpha) = 0, \\ &\\
 0   &  \textrm{if } f(\alpha) \in \mathbb F_{q^n} \textrm{ with }\tr_n f(\alpha) \neq 0. \end{cases}$$
\end{lem}
\begin{proof}
 This is a simple application of Hilbert's Theorem 90.
\end{proof}

Let $\psi_k$, $k=0, \dots, p-1$ be the additive characters of $\F_p$
 given by
$$\psi_k(a) = e^{2 \pi i k a/p}, \quad k=0, \dots, p-1.$$
For each rational function $f \in \F_q(X)$ and non-trivial character $\psi$, we also define
\[
S_n(f, \psi) = \sum_{{x\in \mathbb P^1( \mathbb F_{q^n})}\atop{f(x) \neq \infty}} \psi(\tr_n(f(x))).
\]
Then, using the fact that for any $a \in \F_p$,
$$\sum_{k=0}^{p-1} \psi_k(a) = \begin{cases} p & a = 0, \\ 0 & a \neq 0, \end{cases}$$
it is easy to check that
$$
P_{C_f}(u) = \prod_{\psi \neq \psi_0} L(u, f, \psi)
$$
where
\begin{equation}\label{Euler-product}
L(u,f,\psi) = \exp\left(\sum_{n=1}^{\infty} S_n(f, \psi) \frac{u^n}{n}\right).
\end{equation}

Let $\mathcal{S}=\mathbb F_q[X,Z]$ be the homogeneous coordinate ring of $\mathbb P^1$ and denote  $\mathcal{S}_d$ the $\mathbb F_q$-subspace of $\mathcal{S}$ of homogeneous polynomials of degree $d.$  Notice that $\mathcal{S}_d$ contains the $0$ polynomial and its size is exactly $q^{d+1}.$

Since each Artin-Schreier cover comes equipped with a prescribed map to $\mathbb P^1,$ we can think of $C_f$ as the cover given by
\[C_{g,h}: y^p-y = \frac{g(X,Z)}{h(X,Z)},\]
where the fraction on the right hand side is obtained by homogenizing $f(x)$ in the usual way.

Given $f \in \mathcal{S}_d$, we will denote by $f^*(X) \in \F_q[X]$ the non-homogeneous
polynomial resulting from $f(X,Z)$ by setting $Z=1$. We observe that $f^*$ is polynomial of degree at most $d$. Similarly, let $f_*(Z) \in \F_q[Z]$ be the non-homogeneous
polynomial resulting from $f(X,Z)$ by setting $X=1$.

Given $\alpha=[\alpha_X:\alpha_Z] \in  \PP^1(\mathbb F_{q^k})$ and $h \in \mathcal{S}_d$ the value of $h(\alpha)$ can be zero or nonzero but if it is nonzero, it is not well defined. When we want to discuss
an actual nonzero value we will be talking about $h^*(\alpha):=h(\alpha_X/\alpha_Z,1)$ which is defined for $\alpha \not = [1:0]=\infty$ and $h_*(\alpha):=h(1,\alpha_Z/\alpha_X)$ which is defined for $\alpha \not = [0:1]=0$.

We recall that the rational function $\frac{g}{h}$ can be evaluated in $[\alpha_X:\alpha_Z]$ as long as $g,h \in \mathcal{S}_d$ and\\ $(g(\alpha_X,\alpha_Z),h(\alpha_X,\alpha_Z))\not = (0,0)$.

We now proceed to explicitly describe the families to be considered.
The ordinary case occurs when the $p$-rank is maximal, in other words, when $r$ is maximal. For a given genus $\mathfrak{g}$, this happens when $d_i=1$ in formula \eqref{genusformula} and $2\mathfrak{g}=(p-1)2r$. (Notice once again that this imposes a restriction on the possible values for the genus, as $2\mathfrak{g}/(p-1)$ must be even.) Thus, $f(x)$ is a rational function with exactly $r+1$ simple poles. This corresponds to the fact that $g(X,Z)$ and $h(X,Z)$ are both homogeneous polynomials of degree $r+1$ with no common factors and $h(X,Z)$ is square-free.

We let
\[
\mathcal F_d^{\ord}= \left\{ (g(X,Z),h(X,Z)) : g(X,Z), h(X,Z)\in \mathcal{S}_{d}, h \text{ square-free}, (g,h)=1 \right\},
\]
with the understanding that $d=r+1$.

As $(g,h)$ range over $\mathcal F_d^{\ord},$ the cover $C_{g,h}$ ranges over each $\mathbb F_q$-point of $\mathcal{AS}_{\mathfrak{g},\mathfrak{g}}$ exactly $q-1$ times. Thus, our problem becomes the study of statistics for $C_{g,h}$ as $(g,h)$ varies over $\mathcal F_d^{\ord}$ and $d$ tends to infinity.

We will work with the full family of covers in $\mathcal{AS}_{\mathfrak{g}}$ as well. In this case we do not have the restriction of simple poles but we still require $g(X,Z)$ and $h(X,Z)$ not to have common factors.
\[
\mathcal F_d^{\full}= \left\{ (g(X,Z),h(X,Z)) : g(X,Z), h(X,Z)\in \mathcal{S}_{d}, (g,h)=1 \right\}.
\]
We will then study the statistics as $d$ goes to infinity which is the same as $\mathfrak{g}$ going to infinity provided that the number of poles $r+1$ remains bounded.

Finally, we will consider another family given as follows. We say that $h$ has factorization type $v=(r_1^{d_{1,1}},\dots,r_1^{d_{1,{\ell_1}}}, \dots, r_m^{d_{m,1}},
\dots, r_m^{d_{m,{\ell_m}}})$
if
\[h=P_{1,1}^{d_{1,1}}\cdots P_{1,\ell_1}^{d_{1,\ell_1}} \cdots P_{m,1}^{d_{m,1}}\cdots P_{m,\ell_m}^{d_{m,\ell_m}},\]
where the $P_{i,j}$ are distinct irreducible polynomials of degree $r_i$ and $r_i \not = r_j$ if $i\not = j$. Thus the degree of $h$ is given by $d=\sum_{i=1}^m r_i \sum_{j=1}^{\ell_i}d_{i,j}$.

Let
\[\mathcal{F}_d^v=\{(g(X,Z),h(X,Z)) : g(X,Z), h(X,Z)\in \mathcal{S}_{d}, (g,h)=1, h \mbox{ has factorization type } v\}.\]
In this case, formula \eqref{genusformula} implies $2\mathfrak{g}=(p-1)\left(d-2+\sum_{i=1}^m\ell_i r_i\right)$. Here $\sum_{i=1}^m\ell_i r_i$ represents the number of poles and the $p$-rank is given by $(p-1)\left(\sum_{i=1}^m\ell_i r_i-1\right)$. We will assume the parameters
$m$, $r_i$'s and $\ell_i$'s to be fixed. This implies that the covers considered are all in the same $p$-rank. However, in general, the set of the covers considered does not constitute the whole $p$-rank stratum. We will study the statistics as $d$ goes to infinity which is the same as $\mathfrak{g}$ going to infinity with a bound on the number of poles.

This family includes some important particular cases. Suppose that $v=(1^d)$. This corresponds to the case of only one pole of multiplicity $d$. This pole can always be moved to infinity (i.e., $h(X,Z)=Z^d$). After dehomogenizing with $Z=1$, this
gives the family of $p$-rank 0 covers $\mathcal{AS}_{\mathfrak{g},0}$ indexed by polynomials of degree $d$:
\[\mathcal{F}_d^{{\text{rank}}\, 0}=\{g(x) : \deg(g)=d \}.\]
The statistics for this family were studied in \cite{entin, BDFLS}.

Another interesting case is with $v=(1^{d_1},1^{d_2}).$ In this case we have two poles defined over $\F_q$ that can always be moved to zero and infinity (i.e., $h(X,Z)=X^{d_1}Z^{d_2}$). After dehomogenizing with $Z=1$, this
gives a piece of the family of $p$-rank $p-1$ covers $\mathcal{AS}_{\mathfrak{g},p-1}$ indexed by Laurent polynomials with bidegree $(d_2,d_1)$:
 \[\mathcal{F}_{d_1+d_2}^{{\text{rank}}\, p-1}=\{g(x)/x^{d_1} : \deg(g)=d_2 \}.\]
The other possibility within $p$-rank $p-1$ covers is having two poles defined over $\F_{q^2}\setminus \F_q$, corresponding to $v=(2^d).$ In terms of polynomials, we get, in this case,
\[\mathcal{F}_{2d}^{{\text{rank}}\, p-1}=\{g(x)/h(x)^d : \deg(h)=2, h \mbox{ irreducible}, (g,h)=1 \}.\]
We will show that the statistics for this family is very similar to the statistics for $\mathcal{AS}_{\mathfrak{g},0}$.

We will need to compute the number of elements in a family that satisfy certain values at certain points. The following notation will be useful.
\begin{defn}
 Let $\alpha_1, \dots, \alpha_n, \beta_1,\dots,\beta_n \in \PP^1(\F_{q^k})$. Let $\mathcal{F}_d$ be any of the families under consideration.  We define
\[\mathcal F_{d}(\alpha_1, \dots, \alpha_n, \beta_1, \dots, \beta_n) = \left\{ (g,h) \in \mathcal F_d: (\beta_{i,X}h-\beta_{i,Z}g)(\alpha_i)  = 0, 1 \leq i \leq n \right \}. \]
\end{defn}
We remark that when $\beta\neq \infty$ we identify $\beta=[\beta_X: \beta_Z]$ with $\frac{\beta_X}{\beta_Z}\in \F_{q^k}$ thus
\[
(\beta_{X}h-\beta_{Z}g)(\alpha)  = 0 \Longleftrightarrow \frac{g(\alpha)}{h(\alpha)}=\beta.
\]
A particularly useful case is $\mathcal F_d(\alpha, \beta)$. We remark that this value does not
depend on the value of $\beta$, provided that $\beta\neq \infty$, as we prove below.

\begin{lem}\label{beta=0}
Fix $\alpha\in \PP^1(\F_{q^k})$ of degree $u$ over $\F_q$. Let  $\beta \in \F_{q^u}$. Let $\mathcal{F}_d$ be any of the families under consideration. Then
\[
|\mathcal F_d(\alpha, \beta)|=|\mathcal F_d(\alpha, 0)|.
\]
\end{lem}
\begin{proof} Recall that
\begin{align*}
\mathcal F_d(\alpha, \beta)=\{(g, h) \in \mathcal F_d : (\beta_X h-\beta_Zg)(\alpha)=0 \}.
\end{align*}

Now let $g'= \beta_X h-\beta_Zg$. Since $\beta_Z\not = 0$ we have that $(g,h)=1$ is equivalent to $(g',h)=1$. Then $(g,h) \in \mathcal F_d(\alpha, \beta)$ if and only if $(g',h)\in \mathcal F_d(\alpha,0)$.
\end{proof}

\section{The ordinary case}

In this section, we consider the family
\[
\mathcal F_d^{\ord}= \left\{ (g(X,Z),h(X,Z)) : g(X,Z), h(X,Z)\in \mathcal{S}_{d}, h \mbox{ square-free}, (g,h)=1 \right\}.
\]

\subsection{Heuristics}

We want to calculate, for  given $\alpha=[\alpha_X:\alpha_Z], \beta=[\beta_X:\beta_Z] \in \mathbb P^1(\mathbb F_{q^u})$ such that $\deg \alpha =u$, the probability that
\begin{equation}\label{lincon}
(\beta_Xh - \beta_Zg)(\alpha_X, \alpha_Z)=0\end{equation} as $(g,h)\in \mathcal F_d^{\ord}.$

Locally at $\alpha$ this means that we want  to look at pairs $(g^*,h^*)$ such that $(m_{\alpha}^*)^2 \nmid h^*$ (where $m_{\alpha}^* \in \F_q[X]$ denotes the minimal polynomial of $\alpha$ over $\F_q$) and $(g^*(\alpha), h^*(\alpha))\not\equiv (0,0) \pmod{(m_\alpha^*)^2}$.

 Therefore
\[(g^*,h^*)\equiv \left(\gamma_1 + \delta _1 m_\alpha, \gamma_2 +\delta_2 m_\alpha \right) \pmod{(m_\alpha^*)^2},\]
with $\gamma_i, \delta_i \in \F_q[X]$, and if they are nonzero, $\deg \gamma_i,\deg \delta_i < u$. In addition, the conditions at $\alpha$ imply that $(\gamma_1, \gamma_2) \neq (0,0)$ and  $(\gamma_2, \delta_2) \neq (0,0).$

For each $\gamma_2 \neq 0,$ there are $q^u$ choices for each of the other parameters, thus $q^{3u}(q^u-1)$ total possibilities.
If $\gamma_2=0,$ then there are $q^u-1$ choices for each of $\gamma_1$ and $\delta_2$, and $q^u$ choices for $\delta_1$, for a total of $q^u(q^u-1)^2$ possibilities.

This yields a total of $q^u(q^u-1)(q^{2u}+q^u-1)$ possibilities for $\left(g^*\pmod{(m_\alpha^*)^2}, h^* \pmod{(m_\alpha^*)^2}\right) .$

Now if $\beta = [1:0]=\infty,$ condition \eqref{lincon} reduces to $h^*(\alpha)=0 \iff \gamma_2=0.$ This leaves $q^u-1$ choices for $\gamma_1$ and $\delta_2$ respectively and $q^u$ choices for $\delta_1.$ Thus the probability that $g/h \in \FF_d^{\ord}$ takes the value $\infty$ at a given point $\alpha$ is
\[\frac{q^u(q^u-1)^2}{q^u(q^u-1)(q^{2u}+q^u-1)} = \frac{q^{-u}(1-q^{-u})}{1+q^{-u}-q^{-2u}}.\]

In all other cases, including $\beta=0,$ we must have $h^*(\alpha) \neq 0.$ So there are $q^u-1$ choices for $\gamma_2.$ Once we know $\gamma_2,$ equation \eqref{lincon} fixes $\gamma_1(\alpha)$ (and therefore $\gamma_1$, since its degree is less than $u$), and we have $q^u$ choices for each of $\delta_1, \delta_2.$  Thus the probability $g/h \in \FF_d^{\ord}$ takes the value $\beta \neq \infty$ at a given point $\alpha$ is
\[\frac{q^{2u}(q^u-1)}{q^u(q^u-1)(q^{2u}+q^u-1)} = \frac{q^{-u}}{1 + q^{-u}-q^{-2u}}.\]

Then, the heuristic confirms the result of Proposition \ref{prop:5.6/4.3}, and the expected number of points
of Theorem \ref{thmnumberofpoints}
for the family ${\mathcal{F}}_d^{\ord}$.

\subsection{The number of covers with local conditions} In this subsection, we are going to compute the proportion of polynomials with certain fixed values. We will obtain
the size of the family and the expected number of points as corollaries.

Unless otherwise indicated, we fix $\alpha_1, \dots, \alpha_n \in \mathbb{P}^1(\F_{q^k})$ of degrees $u_1, \dots, u_n$
over $\F_q$ and $\beta_i \in \F_{q^{u_i}}$ for $1 \leq i \leq n$ (i.e. none of the $\beta_i$'s is $\infty$).
Also, $\beta_1, \dots, \beta_\ell$ are not zero, and $\beta_{\ell+1}= \dots = \beta_n=0$.
Finally, none of the $\alpha_i$ are conjugate to each other, i.e. all the minimal polynomials $m_{\alpha_i}$ are distinct.

We start by making the following observation.

\begin{rem}\label{linmap}
If $\alpha=[\alpha:1] \in \mathbb F_{q^{k}}$ has degree $u$ over $\F_q,$ then the map $\mathcal{S}_d \to \F_{q^u}, h \mapsto h^*(\alpha)$ is a linear map of $\F_q$-vector spaces. The map is surjective as long as $d \geq u,$ and in this case its kernel has dimension $d+1-u.$
If $d<u$ the elements $1, \alpha, \alpha^2, \ldots, \alpha^d$ are linearly independent over $\F_q.$ Therefore the image has dimension $d+1$ and thus the kernel has dimension $0.$ In other words the map is injective and the preimage of any element is either empty or a point.

If $\alpha = [1:0]=\infty,$ then it has degree $1$ over $\F_q$ and a condition fixing a value for $h(\alpha)$ can be rewritten in terms of $h_*(1)$ such that it does become linear and the reasoning above applies.
\end{rem}

\begin{lem} \label{lemma: general}
Fix $\alpha_1, \dots, \alpha_n \in \PP^1(\F_{q^k})$ of degrees $u_1, \dots, u_n$ over $\F_q$
such that none of the $\alpha_i$ are conjugate to each other,
and $\beta_i \in \F_{q^{u_i}}$ for $1 \leq i \leq n$ such that
$\beta_1, \dots, \beta_\ell$ are not zero, and $\beta_{\ell+1}= \dots = \beta_n=0$.
Fix $g \in \mathcal{S}_d$ such that $g(\alpha_i) = 0$ for $\ell+1 \leq i \leq n$, and $g(\alpha_i) \neq 0$ for $1 \leq i \leq \ell$. Then we have
\[
\left | \left\{ h \in \mathcal{S}_d : (\beta_{i,X}h-\beta_{i,Z}g)(\alpha_i)  = 0, 1 \leq i \leq n
\right\}\right| =  \begin{cases}
q^{d+1-\sum_{i=1}^\ell u_i} & d \geq \sum_{i=1}^\ell u_i, \\\\
\textrm{$0$ or $1$} & \textrm{otherwise}. \end{cases}
\]
\end{lem}
\begin{proof}
 For $\beta_i\not =0$, the condition imposed over $h$ is $h(\alpha_i)=\frac{g(\alpha_i)}{\beta_i}$ while
there is no condition imposed if $\beta_i=0$. By the Chinese Remainder Theorem, imposing all the conditions together for $\alpha_1,\dots,\alpha_\ell$ is the same as imposing a condition for
$h$ modulo the product $m_{\alpha_1}\cdots m_{\alpha_\ell}$. The result then follows from Remark \ref{linmap}.
\end{proof}

Let $D \in \mathcal{S}_d$. In all the following, the notation $(D)$ means the ideal generated by the polynomial $D.$

\begin{lem}\label{the-one-before}
Fix $\alpha_1, \dots, \alpha_n \in \mathbb P^1(\F_{q^k})$ of degrees $u_1, \dots, u_n$ over $\F_q$
such that none of the $\alpha_i$ are conjugate to each other,
and $\beta_i \in \F_{q^{u_i}}$ for $1 \leq i \leq n$ such that
$\beta_1, \dots, \beta_\ell$ are not zero, and $\beta_{\ell+1}= \dots = \beta_n=0$.
Fix $g \in \mathcal{S}_d$ such that $g(\alpha_i) = 0$ for $\ell+1 \leq i \leq n$, and $g(\alpha_i) \neq 0$ for $1 \leq i \leq \ell$.  Then we have
for any $\varepsilon > 0$
\[\left | \left\{ h \in \mathcal{S}_d :  (h,g)=1, \; \frac{g(\alpha_i)}{h(\alpha_i)} = \beta_i, 1 \leq i \leq n
\right\}\right| = q^{d+1-\sum_{i=1}^\ell u_i} \prod_{(P) | (g)} (1-|P|^{-1}) + O\left(q^{\varepsilon d}\right).\]
If $g(\alpha_i) \neq 0$ for some $\ell+1 \leq i \leq n$, or
$g(\alpha_i) = 0$ for some $1 \leq i \leq \ell$
then the above set is empty.
\end{lem}

\begin{proof} 
If $g(\alpha_i) \neq 0$ for some $\ell+1 \leq i \leq n$, or
$g(\alpha_i) = 0$ for some $1 \leq i \leq \ell$, then it is clear that the above set is empty.
We then suppose $g(\alpha_i) = 0$ for $\ell+1 \leq i \leq n$, and $g(\alpha_i) \neq 0$ for $1 \leq i \leq \ell$.

By inclusion-exclusion and Lemma \ref{lemma: general}  we have
\begin{eqnarray*}
\left|\left\{h \in \mathcal{S}_d : (h,g)=1,   \frac{g(\alpha_i)}{h(\alpha_i)} = \beta_i \right\}\right|
&=&\sum_{(D) | (g)} \mu(D) \sum_{{h\in \mathcal{S}_d}\atop{D| h},  \frac{g(\alpha_i)}{h(\alpha_i)} = \beta_i, 1 \leq i \leq
\ell } 1 \\
&=&  \sum_{{(D) | (g)}\atop{\deg{D} \leq d - \sum_{i=1}^\ell u_i}} \mu(D) q^{d+1-\deg{D}-\sum_{i=1}^\ell u_i} + \sum_{{(D) | (g)}\atop {d - \sum_{i=1}^\ell u_\ell < \deg{D} \leq d}} O(1) \\
&=&  q^{d+1-\sum_{i=1}^\ell u_i} \sum_{{(D) | (g)}} \mu(D) q^{-\deg{D}} + \sum_{{(D) | (g)}\atop {d - \sum_{i=1}^\ell u_\ell < \deg{D} \leq d}} O(1) \\
&=&   q^{d+1-\sum_{i=1}^\ell u_i} \prod_{(P)\mid (g)} (1-|P|^{-1}) +  O \left( q^{\varepsilon d}\right)
\end{eqnarray*}
where $\mu$ is the M\"obius function.

\end{proof}
\begin{defn}
 Let $g \in \mathcal{S}_{d}$. Set
\[
A_{d}^g= \{ h \in \mathcal{S}_{d}: h \text{ square free  and } (h,g)=1\}.\]
Let $\alpha_1, \dots, \alpha_n, \beta_1,\dots,\beta_n \in \PP^1(\F_{q^k})$. We define
\[A_d^g(\alpha_1, \dots, \alpha_n, \beta_1, \dots, \beta_n) = \left\{ h \in A_{d}^g: (\beta_{i, X} h-\beta_{i,Z}g)(\alpha_i)=0, 1\leq i \leq n \right\}. \]
\end{defn}

\begin{lem}\label{lem:gcoprimesqfreeval}
Fix $\alpha_1, \dots, \alpha_n \in \mathbb P^1(\F_{q^k})$ of degrees $u_1, \dots, u_n$ over $\F_q$ such that none of the $\alpha_i$ are conjugate to each other. Let $\beta_i \in \F_{q^{u_i}}$ for $1 \leq i \leq n$ such that
$\beta_1, \dots, \beta_\ell$ are not zero, and $\beta_{\ell+1}= \dots = \beta_n=0$.
Fix $g \in \mathcal{S}_{d}$ such that $g(\alpha_i)=0$ for $\ell+1 \leq i\leq n$ and $g(\alpha_i)\not = 0$ for $1 \leq i\leq \ell$. Then
\[|A_d^g(\alpha_1, \dots, \alpha_n, \beta_1, \dots, \beta_n) | =
 \frac{q^{d+1-\sum_{i=1}^\ell u_i}}{\zeta_q(2) \prod_{i=1}^\ell(1-q^{-2u_i})} \prod_{(P)\mid (g)}  (1+|P|^{-1})^{-1} + O\left(q^{(1/2 + \varepsilon)d}\right).
 \]
If $g(\alpha_i) \neq 0$ for some $\ell+1 \leq i \leq n$, or
$g(\alpha_i) = 0$ for some $1 \leq i \leq \ell$
then the above set is empty.
\end{lem}

\begin{proof} It is clear that  $A_d^g(\alpha_1, \dots, \alpha_n, \beta_1, \dots, \beta_n) $ is empty if the condition on the values $g(\alpha_i)$ of the lemma are not satisfied, and we then suppose that $g(\alpha_i) = 0$ for $\ell+1 \leq i \leq n$, and $g(\alpha_i) \neq 0$ for $1 \leq i \leq \ell$.

By inclusion-exclusion,
 \begin{eqnarray*} |A_d^g(\alpha_1, \dots, \alpha_n, \beta_1, \dots, \beta_n) |
 &=& {\sideset{}{'}\sum_{(D): (D,g)=1 \atop {\deg(D)\leq d/2}}}\mu(D) \left|\left\{h_1\in \mathcal{S}_{d-2\deg(D)}:  (h_1,g)=1, \frac{g(\alpha_i)}{h_1(\alpha_i)} = D^2(\alpha_i) \beta_i
 \right\} \right| \\
  &=& q^{d+1-\sum_{i=1}^\ell u_i} \prod_{(P)\mid (g)} (1-|P|^{-1})
  \sideset{}{'}\sum_{(D): (D,g)=1 \atop {\deg(D) \leq d/2}} \mu(D)|D|^{-2}
  + \sideset{}{'}\sum_{{(D): (D,g)=1} \atop {\deg{D} \leq d/2}} O \left( q^{\varepsilon d} \right)\\
 \end{eqnarray*}
by Lemma \ref{the-one-before}, where we have written $\sideset{}{'}\sum_{(D)}$ for the sum over polynomials $D$ such that
$D(\alpha_i) \neq 0$ for $1 \leq i \leq \ell.$

But \[\sideset{}{'}\sum_{(D): (D,g)=1}\mu(D)|D|^{-2s} = \prod_{(P): P\nmid g \atop{P(\alpha_i)\neq 0, 1 \leq i \leq \ell}} (1-|P|^{-2s}) = \prod_{(P): P\nmid g m_{\alpha_1} \dots m_{\alpha_\ell}} (1-|P|^{-2s}),
\]
where we made use of the fact that $(g, m_{\alpha_i})=1$ since $g(\alpha_i)\neq 0.$ This can be rewritten as
\[\frac{1}{\zeta_q(2s)} \prod_{(P)\mid (gm_{\alpha_1} \dots m_{\alpha_\ell})}  (1-|P|^{-2s})^{-1} = \frac{1}{\zeta_q(2s)\;\prod_{i=1}^\ell ( 1-q^{-2su_i})} \prod_{(P)\mid (g)}  (1-|P|^{-2s})^{-1}   .\]

Therefore
\[\sideset{}{'}\sum_{(D): (D,g)=1 \atop {\deg(D) \leq d/2}} \mu(D)|D|^{-2}= \frac{1}{\zeta_q(2)\;\prod_{i=1}^\ell ( 1-q^{-2u_i})} \prod_{(P)\mid (g)}  (1-|P|^{-2})^{-1} +O\left( q^{-d/2} \right)\]
and
\begin{eqnarray*}
|A_d^g(\alpha_1, \dots, \alpha_n, \beta_1, \dots, \beta_n) | &=&
\frac{q^{d+1-\sum_{i=1}^\ell u_i}}{\zeta_q(2)\prod_{i=1}^\ell (1-q^{-2u_i})} \prod_{(P)\mid (g)}  (1+|P|^{-1})^{-1} + O \left(q^{(1/2 + \varepsilon)d} \right).
\end{eqnarray*}

\end{proof}

\begin{prop}\label{prop:frvalue}
Fix $\alpha_1, \dots, \alpha_n \in \mathbb P^1(\F_{q^k})$ of degrees $u_1, \dots, u_n$ over $\F_q$ such that none of the $\alpha_i$ are conjugate to each other. Let $\beta_i \in \F_{q^{u_i}}$ for $1 \leq i \leq n$.  Then
\[|\mathcal F_{d}^{\ord}(\alpha_1, \dots, \alpha_n, \beta_1, \dots, \beta_n) | =
\frac{H(1)q^{2d+2-\sum_{i=1}^n u_i}}{\zeta_q(2)^2\prod_{i=1}^n (1+q^{-u_i}-q^{-2u_i})}+
O\left(q^{(3/2 + \varepsilon)d}\right),
\]
where
\[H(1)= \prod_{(P)}\left(1+\frac{1}{(|P|+1)(|P|^2-1)}\right).\]
\end{prop}

\begin{proof} Denote by $m_{\alpha_i}$ the homogenized minimal polynomial of $\alpha_i$ over $\F_q.$
We have
\[ |\mathcal F_{d}^{\ord}(\alpha_1,\dots, \alpha_n,\beta_1, \dots, \beta_n) | = \sum_{g \in \mathcal{S}_{d}} |A_d^g(\alpha_1,\dots, \alpha_n,\beta_1, \dots, \beta_n)|.\] Assume without loss of generality that
$\beta_1, \dots, \beta_\ell$ are not zero, and $\beta_{\ell+1}= \dots = \beta_n=0$. By Lemma \ref{lem:gcoprimesqfreeval}, the above sum equals
\begin{eqnarray*}
|\mathcal F_{d}^{\ord}(\alpha_1, \dots, \alpha_n, \beta_1, \dots, \beta_n) | =  \sum_{{{g \in \mathcal{S}_{d}} \atop {g(\alpha_i)\neq 0, 1\leq i \leq \ell}} \atop {g(\alpha_i)=0, \ell+1 \leq i \leq n}}
\left(\frac{q^{d+1-\sum_{i=1}^\ell u_i}}{\zeta_q(2) \prod_{i=1}^\ell(1-q^{-2 u_i})} \prod_{(P) \mid (g)} (1 + |P|^{-1})^{-1}+O\left(q^{(1/2 + \varepsilon)d}\right) \right)  \\
 = \frac{q^{d+1-\sum_{i=1}^n u_i}}{\zeta_q(2)\prod_{i=1}^\ell (1-q^{-2 u_i})}
 \sum_{{{g \in \mathcal{S}_{d}} \atop {g(\alpha_i)\neq 0, 1\leq i \leq \ell}} \atop {g(\alpha_i)=0, \ell+1 \leq i \leq n}}
 \prod_{(P) \mid (g)} (1 + |P|^{-1})^{-1}+ O\left(q^{(3/2 + \varepsilon)d}\right). \end{eqnarray*}

Set
\[b(g)=\prod_{(P) \mid (g)} (1 + |P|^{-1})^{-1}\]
and
\[G(s)=\sum_{(g)\not = 0} \frac{b(g)}{|g|^s}.\]

Since $b(g)$ is a multiplicative function, it follows that $G(s)$ has an Euler product of the form
\begin{eqnarray*}
 G(s) &=& \prod_{(P)} \left(  \sum_{k=0}^\infty b(P^k) |P|^{-ks} \right)\\
& =& \prod_{(P)} \left( 1+ \frac{b(P)|P|^{-s}}{1- |P|^{-s}}\right)\\
&  =& \prod_{(P)} \left( 1+ \frac{|P|^{-s}}{(1- |P|^{-s})(1+ |P|^{-1})}\right).
\end{eqnarray*}
Thus \[G(s) = \frac{\zeta_q(s)}{\zeta_q(2s)} H(s),\] where \[H(s)  =\prod_{(P)} \left(1-\frac{|P|^{-s}(1-|P|^{1-s}-|P|^{-s})}{(|P|+1)(1-|P|^{-2s})} \right),\]
which converges for $\mathrm{Re}(s)>1/2.$  In addition, $G(s)$ has a simple pole at $s=1$ with residue
\[\frac{H(1)}{\zeta_q(2)\log q}=\frac{1}{\zeta_q(2)\log q} \prod_{(P)} \left( 1 +\frac{1}{(|P|+1)(|P|^2-1)}\right).\]

Define the additional Dirichlet series
\begin{eqnarray*}
G_1(s) &=&  \sum_{{(m_{\alpha_i}) \nmid (g), 1 \leq i \leq \ell}\atop{(m_{\alpha_i}) \mid (g), \ell+1 \leq i \leq n}}
 \frac{b(g)}{|g|^s} = \prod_{(P)\neq (m_{\alpha_i}), 1\leq i \leq n} \left( 1+ \frac{|P|^{-s}}{(1- |P|^{-s})(1+ |P|^{-1})}\right) \\
 && \times \prod_{(P) = (m_{\alpha_i}), \ell+1\leq i \leq n} \left(\sum_{k=1}^\infty b(P^k)|P|^{-ks}\right)
 \\
&=& G(s)
\prod_{i=1}^n
\left( 1 + \frac{q^{-u_is} }{(1-q^{-u_is})(1+q^{-u_i})}\right)^{-1}\prod_{i=\ell+1}^n \frac{q^{-u_is}}{(1-q^{-u_is})(1+q^{-u_i})}\\
&=& G(s)
\prod_{i=1}^\ell \frac{(1-q^{-u_is})(1+q^{-u_i})}{1+q^{-u_i}-q^{-u_i(s+1)}}\prod_{i=\ell+1}^n \frac{q^{-u_is}}{1+q^{-u_i}-q^{-u_i(s+1)}}.\\
\end{eqnarray*}

Thus, $G_1(s)$ has a simple pole at $s=1$ with residue
\[ \rho = \frac{H(1)}{\zeta_q(2)\log q} \prod_{i=1}^\ell  \frac{1-q^{-2u_i}}{1+q^{-u_i}-q^{-2u_i}}
\prod_{i=\ell+1}^n \frac{q^{-u_i}}{1+q^{-u_i}-q^{-2u_i}}, \]
and $$G_1(s) - \frac{\rho}{s-1}$$ is holomorphic for $\mbox{Re}(s) > 1/2$.
Then, using Theorem 17.1 of \cite{rosen}
which is the function field version of the Wiener--Ikehara Tauberian Theorem, we get that
\begin{eqnarray*}
 \sum_{{(g), g \in \mathcal{S}_{d}}\atop {{(m_{\alpha_i}) \nmid (g), 1 \leq i \leq \ell}\atop{(m_{\alpha_i}) \mid (g), \ell+1 \leq i \leq n}} }
b(g) = \frac{H(1) q^{d+1}}{\zeta_q(2)} \prod_{i=1}^\ell  \frac{1-q^{-2u_i}}{1+q^{-u_i}-q^{-2u_i}}
\prod_{i=\ell+1}^n \frac{q^{-u_i}}{1+q^{-u_i}-q^{-2u_i}} + O \left( q^{(1/2+\varepsilon)d} \right).
\end{eqnarray*}
Using the line above in the formula for $|\mathcal F_{d}^{\ord}(\alpha_1, \dots, \alpha_n, \beta_1, \dots, \beta_n) |$, we get
\begin{eqnarray*}
&&|\mathcal F_{d}^{\ord}(\alpha_1, \dots, \alpha_n, \beta_1, \dots, \beta_n) | \\
&&=
 \frac{q^{d+1-\sum_{i=1}^\ell u_i}}{\zeta_q(2) \prod_{i=1}^\ell
(1-q^{-2 u_i})}
  \frac{H(1)q^{d+1}}{\zeta_q(2)}\prod_{i=1}^\ell  \frac{1-q^{-2u_i}}{1+q^{-u_i}-q^{-2u_i}}
\prod_{i=\ell+1}^n \frac{q^{-u_i}}{1+q^{-u_i}-q^{-2u_i}}+ O\left(q^{(3/2 + \varepsilon)d}\right) \\
&& =\frac{H(1)q^{2d+2-\sum_{i=1}^n u_i}}{\zeta_q(2)^2
{\prod_{i=1}^n (1+q^{-u_i}-q^{-2u_i})}}
+ O\left(q^{(3/2 + \varepsilon)d}\right).
\end{eqnarray*}

\end{proof}

The previous result may be used to obtain the number of covers in the whole ordinary family by specializing to $n=0$.

\begin{cor}\label{prop:sizefamily}
 \[|\mathcal{F}_d^{\ord}| =\frac{H(1)q^{2d+2}}{\zeta_q(2)^2}+ O\left(q^{(3/2 + \varepsilon)d}\right).\]
\end{cor}

By combining Proposition \ref{prop:frvalue} and Corollary \ref{prop:sizefamily}, we obtain the following result.
\begin{prop} \label{nalpha} Fix $\alpha_1, \dots, \alpha_n \in \mathbb P^1(\F_{q^k})$ of degrees $u_1, \dots, u_n$ over $\F_q$ such that none of the $\alpha_i$ are conjugate to each other. Let $\beta_i \in \F_{q^{u_i}}$ for $1 \leq i \leq n$.
 Then
\begin{eqnarray*} \frac{|\mathcal{F}_d^{\ord}(\alpha_1, \dots, \alpha_n,\beta_1, \dots, \beta_n)|}{|\mathcal{F}_d^{\ord}|}
&=&
\frac{q^{-\sum_{i=1}^n u_i}}{\prod_{i=1}^n (1+q^{-u_i}-q^{-2u_i})}+ O\left(q^{(-1/2+\varepsilon)d}\right)  \\
&=& {q^{-\sum_{i=1}^n u_i}} \left( 1 + O \left( \sum_{i=1}^n q^{-u_i} \right) \right)+   O\left(q^{(-1/2+\varepsilon)d}\right).
\end{eqnarray*}
\end{prop}

We finish this section by computing the expected number of points in an ordinary Artin-Schreier cover. For this, we need to compute the case $n=1$,
 i.e., $|\mathcal F_d^{\ord}(\alpha, \beta)|$.

\begin{cor} \label{1alpha}
 Fix $\alpha \in \PP^1(\F_{q^k})$ of degree $u$ over $\F_q$. Let $\beta \in \PP^1(\F_{q^u})$. Then
\[|\mathcal F_{d}^{\ord}(\alpha, \beta) | = \begin{cases}
\frac{H(1)q^{2d+2-u}(1-q^{-u})}{\zeta_q(2)^2(1+q^{-u}-q^{-2u})} +  O\left(q^{(3/2+\varepsilon)d+u}\right) & \beta=\infty,\\\\
\frac{H(1)q^{2d+2-u}}{\zeta_q(2)^2(1+q^{-u}-q^{-2u})} + O\left(q^{(3/2+\varepsilon)d}\right) & \beta \in  \F_{q^u}.
\end{cases}\]
\end{cor}
\begin{proof} The case of $\beta \in  \F_{q^u}$ is a simple consequence of Proposition \ref{prop:frvalue}. For $\beta=[1:0]$, we have, by Lemma \ref{beta=0}, that
\begin{eqnarray*}
|\mathcal F_d^{\ord}(\alpha, \infty)|&=&|\mathcal F_d^{\ord}|-\sum_{\beta \in  \F_{q^u}}|\mathcal F_d^{\ord}(\alpha, \beta)|\\
&=&|\mathcal F_d^{\ord}|-q^u|\mathcal F_d^{\ord}(\alpha,0)|\\
&=&\frac{H(1)q^{2d+2-u}(1-q^{-u})}{\zeta_q(2)^2(1+q^{-u}-q^{-2u})} + O\left(q^{(3/2+\varepsilon)d+u}\right).
 \end{eqnarray*}

\end{proof}

By combining Proposition \ref{nalpha} and Corollaries \ref{prop:sizefamily} and \ref{1alpha}, we obtain the following result.

\begin{prop} \label{prop:5.6/4.3} Fix $\alpha \in \PP^1(\F_{q^k})$ with degree $u$ over $\F_q$. Let $\beta \in \PP^1(\F_{q^u})$.  Then
\[\frac{|\mathcal{F}_d^{\ord}(\alpha,\beta)|}{|\mathcal{F}_d^{\ord}|}=\begin{cases}
\frac{q^{-u}(1-q^{-u})}{1+q^{-u}-q^{-2u}} +  O\left(q^{(-1/2+\varepsilon)d+u}\right) & \beta=\infty,\\\\
\frac{q^{-u}}{1+q^{-u}-q^{-2u}} + O\left(q^{(-1/2+\varepsilon)d}\right) & \beta \in  \F_{q^u}.
\end{cases}\]
\end{prop}

\begin{lem} \label{lem:average}  Fix $\alpha \in \PP^1(\F_{q^k})$ of degree $u$ over $\F_q$. The expected number of  $\F_{q^k}$-points in the fiber above $\alpha$ is
\[\begin{cases} 1 +  O\left(q^{(-1/2+\varepsilon)d+u}\right) & \textrm{ if } p\nmid \frac{k}{u}, \\
&\\
1
+\frac{p-1}{1+q^{-u}-q^{-2u}} +  O\left(q^{(-1/2+\varepsilon)d+u}\right)& \textrm{ if }  p\mid \frac{k}{u}. \\

\end{cases}\]
\end{lem}
\begin{proof}
 By Lemma \ref{expectednumberfiber} and Proposition \ref{prop:5.6/4.3}, the expected number of $\F_{q^k}$-points in the fiber above $\alpha$ is
 \[\frac{q^{-u}(1-q^{-u})}{1+q^{-u}-q^{-2u}} + O\left(q^{(-1/2+\varepsilon)d+u}\right)\\
+\sum_{\beta \in \F_{q^u}, \tr_k(\beta)=0}p\left(\frac{q^{-u}}{1+q^{-u}-q^{-2u}} + O\left(q^{(-1/2+\varepsilon)d}\right) \right).\]

If $p\nmid \frac{k}{u}$, then $\tr_k(\beta)=0$ iff $\tr_u(\beta)=0$ and there are $\frac{q^u}{p}$ points in $\F_{q^u}$ with that property.

If $p \mid \frac{k}{u}$, then $\tr_k(\beta) = \frac{k}{u} \tr_u(\beta) = 0$ for all $\beta \in \F_{q^u}$  and therefore the expected number of points in the fiber is

  \[\frac{q^{-u}(1-q^{-u})}{1+q^{-u}-q^{-2u}} +  O\left(q^{(-1/2+\varepsilon)d+u}\right)\\
+\frac{p}{1+q^{-u}-q^{-2u}} +  O\left(q^{(-1/2+\varepsilon)d+u}\right).\]

\end{proof}

For our main result, we recall that an ordinary Artin-Schreier cover has $r+1$ simple poles. This corresponds to taking $d=r+1$. We are ready to prove the first part of Theorem
\ref{thmnumberofpoints}.

\begin{thm} \label{expected-ord}
The expected number of $\F_{q^k}$-points on an ordinary Artin-Schreier cover defined over $\F_q$ is
\[\begin{cases} q^{k} +1 + O\left(q^{(-1/2+\varepsilon)(r+1)+2k}\right)& p \nmid k, \\\\
 q^{k} +1 +\frac{p-1}{1+q^{-1}-q^{-2}} +\sum_{u\mid \frac{k}{p}} \frac{p-1}{1+q^{-u}-q^{-2u}} \pi(u)u+  O\left(q^{(-1/2+\varepsilon)(r+1)+2k}\right)& p \mid k,\end{cases}\]
where $\pi(u)$ is the number of monic irreducible polynomials in $\F_q[X]$ of degree $u$.
 \end{thm}
\begin{proof}
 If $p \nmid k$, the result follows by adding  the result of Lemma \ref{lem:average} over all $\alpha \in \PP^1(\F_{q^k})$. 
 If $p\mid k$ we still get the term $q^k+1$ and an additional term given by
\begin{align*}
 \sum_{u\mid \frac{k}{p}} \sum_{\alpha, \deg \alpha=u} \frac{p-1}{1+q^{-u}-q^{-2u}} & = \frac{p-1}{1+q^{-1}-q^{-2}}+
\sum_{u\mid \frac{k}{p}} \frac{p-1}{1+q^{-u}-q^{-2u}} \pi(u)u,
\end{align*}
where the first term on the right hand side accounts for the case $\alpha=\infty$.
\end{proof}

\begin{rem}
 When $k=p$, we obtain
\[q^{p} +1 +\frac{(p-1)(q+1)}{1+q^{-1}-q^{-2}} +  O\left(q^{(-1/2+\varepsilon)(r+1)+2p}\right).\]
\end{rem}

\section{Full Space}
In this case, we consider the family
\[
\mathcal F_d^{\full}= \left\{ (g(X,Z),h(X,Z)) : g(X,Z), h(X,Z)\in \mathcal{S}_{d}, (g,h)=1 \right\}.
\]

\begin{prop}\label{full:nalpha} Fix $\alpha_1, \dots, \alpha_n \in \mathbb P^1(\F_{q^k})$ of degrees $u_1, \dots, u_n$ such that none of the $\alpha_i$ are conjugate to each other. Let $\beta_i \in \F_{q^{u_i}}$ for $1 \leq i \leq n$.
 Then we have
\[ |\mathcal  F_d^{{\full}}(\alpha_1, \dots, \alpha_n,\beta_1, \dots, \beta_n)|=\frac{q^{2d+2-\sum_{i=1}^n u_i}}{\zeta_q(2)\prod_{i=1}^n\left(1+q^{-u_i}\right)}+O\left(q^{(1+\varepsilon)d}\right).\]
\end{prop}
\begin{proof} Assume without loss of generality that
$\beta_1, \dots, \beta_\ell$ are not zero, and $\beta_{\ell+1}= \dots = \beta_n=0$. We have, by Lemma \ref{the-one-before}, that
\begin{eqnarray*}
|\mathcal  F_d^{\full}(\alpha_1, \dots, \alpha_n,\beta_1, \dots, \beta_n)|&=&\sum_{g \in \mathcal{S}_{d}} \left|\left\{h \in \mathcal{S}_{d}: (h,g)=1, \frac{g(\alpha_i)}{h(\alpha_i)}=\beta_i, 1\leq i \leq n\right\}\right|\\
&=& \sum_{g \in \mathcal{S}_d} q^{d+1-\sum_{i=1}^\ell u_i} \prod_{(P) | (g)} (1-|P|^{-1}) + O\left(q^{(1+\varepsilon) d}\right).
\end{eqnarray*}

We set
\[b(g)=\prod_{(P) \mid (g)} (1 - |P|^{-1}),\]
and
\[G(s)=\sum_{(g)\not = 0} \frac{b(g)}{|g|^s}.\]
Since $b(g)$ is a multiplicative function, it follows that $G(s)$ has an Euler product of the form
\begin{eqnarray*}
 G(s)&=&\prod_{(P)}\left(\sum_{k=0}^\infty b(P^k)|P|^{-ks}\right)
=\prod_{(P)}\left(1+\frac{b(P)|P|^{-s}}{1-|P|^{-s}}\right)\\
&=&\prod_{(P)}\left(1+\frac{(1-|P|^{-1})|P|^{-s}}{1-|P|^{-s}}\right)
=\prod_{(P)}\left(\frac{1-|P|^{-1-s}}{1-|P|^{-s}}\right).
\end{eqnarray*}
Therefore
\[G(s)=\frac{\zeta_q(s)}{\zeta_q(1+s)},\]
is analytic for $\mathrm{Re}(s)>0$, except for a simple pole at $s=1$ with residue $\frac{1}{\zeta_q(2)\log q }$.

Now define the Dirichlet series
\begin{eqnarray*}
G_1(s) &=&  \sum_{{(m_{\alpha_i}) \nmid (g), 1 \leq i \leq \ell}\atop{(m_{\alpha_i}) \mid (g), \ell+1 \leq i \leq n}}
 \frac{b(g)}{|g|^s} = \prod_{(P)\neq ( m_{\alpha_i}), 1\leq i \leq n} \left(\frac{1-|P|^{-1-s}}{1-|P|^{-s}}\right) \\
 &&\times  \prod_{(P) = (m_{\alpha_i}), \ell+1\leq i \leq n} \left(\sum_{k=1}^\infty b(P^k)|P|^{-ks}\right)\\
 \\
&=&G(s)
\prod_{i=1}^n  \left(\frac{1-q^{-u_i(1+s)}}{1-q^{-u_is}}\right)^{-1} \prod_{i=\ell+1}^n  \left(\frac{q^{-u_is}(1-q^{-u_i})}{1-q^{-u_is}} \right)
\\
&=&G(s)
\prod_{i=1}^\ell \frac{1-q^{-u_is}}{1-q^{-u_i(1+s)}} \prod_{i=\ell+1}^n  \frac{q^{-u_is}(1-q^{-u_i})}{1-q^{-u_i(1+s)}}.
\\
\end{eqnarray*}

Thus $G_1(s)$  is analytic for $\mathrm{Re}(s)>0$, except for a simple pole at $s=1$ with residue
\[\frac{1}{\zeta_q(2)\log q}\prod_{i=1}^\ell\frac{1}{1+q^{-u_i}}\prod_{i=\ell+1}^n  \frac{q^{-u_i}}{1+q^{-u_i}}. \]
Then, using again Theorem 17.1 of \cite{rosen}, we get that
\begin{eqnarray*}
|\mathcal  F_d^{\full}(\alpha_1, \dots, \alpha_n,\beta_1, \dots, \beta_n)|&=&  q^{d+1-\sum_{i=1}^\ell u_i} \sum_{{(g), g \in \mathcal{S}_{d}}\atop{{(m_{\alpha_i}) \nmid (g), 1 \leq i \leq \ell}\atop{(m_{\alpha_i}) \mid (g), \ell+1 \leq i \leq n}}} b(g) + O\left(q^{(1+\varepsilon)d}\right)\\
&=&  \frac{q^{2d+2-\sum_{i=1}^\ell u_i}}{\zeta_q(2)}\prod_{i=1}^\ell\left(\frac{1}{1+q^{-u_i}}\right)\prod_{i=\ell+1}^n  \left(\frac{q^{-u_i}}{1+q^{-u_i}} \right) + O\left(q^{(1+\varepsilon)d}\right).\\
\end{eqnarray*}

\end{proof}

We may now proceed to compute the number of covers in the whole family by setting $n=0$ in the previous result.
\begin{cor}\label{full}
\[ |\mathcal  F_d^{\full}|=\frac{q^{2d+2}}{\zeta_q(2)}+O\left(q^{(1+\varepsilon)d}\right).\]
\end{cor}

By combining Proposition \ref{full:nalpha} and Corollary \ref{full}, we obtain the following result.
\begin{prop}\label{fullalphas} Fix $\alpha_1, \dots, \alpha_n \in \mathbb P^1(\F_{q^k})$ of degrees $u_1, \dots, u_n$ such that none of the $\alpha_i$ are conjugate to each other. Let $\beta_i \in \F_{q^{u_i}}$ for $1 \leq i \leq n$.
 Then we have
\begin{eqnarray*}\frac{|\mathcal F_d^{\full}(\alpha_1, \dots, \alpha_n,\beta_1, \dots, \beta_n)|}{|\mathcal F_d^{\full}|}&=& \frac{q^{-\sum_{i=1}^n u_i}}{\prod_{i=1}^n\left({1+q^{-u_i}}\right)}+O\left(q^{(\varepsilon-1)d}\right)\\
&=& q^{-\sum_{i=1}^n u_i}\left(1+O \left(\sum_{i=1}^nq^{- u_i}\right)\right) +O\left(q^{(\varepsilon-1)d}\right).
 \end{eqnarray*}

\end{prop}

We finish the section by computing the expected number of points in the full Artin-Schreier family.

\begin{cor} Fix $\alpha\in \PP^1(\F_{q^k})$ of degree $u$ over $\F_q$. Let  $\beta \in \PP^1(\F_{q^u})$. Then
 \[|\mathcal  F_d^{\full}(\alpha,\beta)|=\frac{q^{2d+2-u}}{\zeta_q(2)(1+q^{-u})}+\begin{cases} O\left(q^{(\varepsilon+1)d+u}  \right) & \beta=\infty,\\\\
                                                                                        O\left(q^{(\varepsilon+1)d}  \right) & \beta\in \F_{q^u}.
                                                                                      \end{cases}
\]
\end{cor}

\begin{proof} The case of $\beta \in \F_{q^u}$ easily follows from Proposition \ref{full:nalpha}. For $\beta=[1:0]$, we have, by Lemma
\ref{beta=0} that
 \begin{eqnarray*}
 |\mathcal  F_d^{\full}(\alpha,\infty)|&=& |\mathcal  F_d^{\full}| - \sum_{\beta \in \F_{q^u}}|\mathcal  F_d^{\full}(\alpha,\beta)|\\
&=&|\mathcal  F_d^{\full}| - q^u |\mathcal  F_d^{\full}(\alpha, 0)| \\
&=& \frac{q^{2d+2-u}}{\zeta_q(2)(1+q^{-u})}+O\left(q^{(\varepsilon+1)d+u} \right).
 \end{eqnarray*}

\end{proof}

We then obtain the following result.
\begin{prop}\label{quotientfull}
 Fix $\alpha\in \PP^1(\F_{q^k})$ of degree $u$ over $\F_q$. Let  $\beta \in \PP^1(\F_{q^u})$. Then
 \[\frac{|\mathcal  F_d^{\full}(\alpha,\beta)|}{|\mathcal  F_d^{\full}|}=\frac{q^{-u}}{1+q^{-u}}+\begin{cases} O\left(q^{(\varepsilon-1)d+u}  \right) & \beta=\infty,\\\\
                                                                                        O\left(q^{(\varepsilon-1)d}  \right) & \beta\in \F_{q^u}.
                                                                                      \end{cases}
\]
\end{prop}

\begin{lem}\label{fullfiber}Fix $\alpha\in \PP^1(\F_{q^k})$ of degree $u$ over $\F_q$. The expected number of $\F_{q^k}$-points in the fiber above
$\alpha$ is
 \[\begin{cases}
1+ O\left(q^{(\varepsilon-1)d+u}\right) & \mbox{if } \, p \nmid \frac{k}{u},\\\\
1+ \frac{p-1}{1+q^{-u}}+O\left(q^{(\varepsilon-1)d+u}\right) & \mbox{if } \, p\mid \frac{k}{u}.
 \end{cases}\]

\end{lem}
\begin{proof} By Lemma \ref{expectednumberfiber} and Proposition \ref{quotientfull}, we have
 \[\frac{q^{-u}}{1+q^{-u}}+O\left(q^{(\varepsilon-1)d+u}\right)+\sum_{\beta \in \F_{q^u}, \tr_k(\beta)=0}p
  \left(\frac{q^{-u}}{1+q^{-u}}+O\left(q^{(\varepsilon-1)d}\right) \right).
 \]

If $p \nmid \frac{k}{u}$, then $\tr_k(\beta)=0$ iff $\tr_u(\beta)=0$ and there are $\frac{q^u}{p}$ points in $\F_{q^u}$ with that property.

If  $p \mid \frac{k}{u}$, then $\tr_k(\beta)=\frac{k}{u}\tr_u(\beta)=0$ for all $\beta \in \F_{q^u}$ and therefore
the expected number of points in the fiber is
 \[\frac{q^{-u}}{1+q^{-u}}+O\left(q^{(\varepsilon-1)d+u}\right)+\frac{p}{1+q^{-u}}+O\left(q^{(\varepsilon-1)d+u}\right).
 \]
\end{proof}

We are ready to prove Theorem
\ref{thmnumberofpoints} (2).

\begin{thm}\label{thm:allrank}
 The expected number of $\F_{q^k}$-points on an Artin-Schreier cover in $\mathcal{AS}_{\mathfrak{g}}$ defined over $\F_q$ is
\[\begin{cases} q^{k} +1 +O\left(q^{(\varepsilon-1)d+2k}\right)& p \nmid k, \\ \\
q^k+1+(p-1)q^{k/p}+\frac{p-1}{1+q^{-1}}-(p-1)\sum_{u|\frac{k}{p}}\frac{1}{1+q^u}\pi(u)u+O\left(q^{(\varepsilon-1)d+2k}\right)& p \mid k. \end{cases}\]
\end{thm}

\begin{proof} The result for $p \nmid k$ follows from Lemma \ref{fullfiber}. If $p \mid k$ we still get the term $q^k+1$ and an
additional term given by
\begin{eqnarray*}
\sum_{u\mid \frac{k}{p}}\sum_{\alpha, \deg \alpha=u} \frac{p-1}{1+q^{-u}}&=& \frac{p-1}{1+q^{-1}}+(p-1)\sum_{u\mid \frac{k}{p}}\frac{q^u}{1+q^u}\pi(u)u\\
&=&  \frac{p-1}{1+q^{-1}}+(p-1)q^{k/p}-(p-1)\sum_{u\mid \frac{k}{p}}\frac{1}{1+q^u}\pi(u)u.\\
\end{eqnarray*}

\end{proof}

\begin{rem} When $k=p$, we obtain
\[ q^p+1+(p-1)q+O\left(q^{(\varepsilon-1)d+2p}\right).\]
\end{rem}

\section{Prescribed factorization type}\label{sec:prescribed}

Recall that
\[\mathcal{F}_d^v=\{(g(X,Z),h(X,Z)) : g(X,Z), h(X,Z)\in \mathcal{S}_{d}, (g,h)=1, h \mbox{ has factorization type } v\},\]
where $v=(r_1^{d_{1,1}},\dots,r_1^{d_{1,{\ell_1}}}, \dots, r_m^{d_{m,1}}, \dots, r_m^{d_{m,{\ell_m}}})$ and
\[h=P_{1,1}^{d_{1,1}}\cdots P_{1,\ell_1}^{d_{1,\ell_1}} \cdots P_{m,1}^{d_{m,1}}\cdots P_{m,\ell_m}^{d_{m,\ell_m}},\]
where the $P_{i,j}$ are distinct irreducible polynomials of degree $r_i$ and $r_i \not = r_j$ if $i \not = j$. The degree of $h$ is then given by $d=\sum_{i=1}^m r_i \sum_{j=1}^{\ell_i}d_{i,j}$.

We will first compute the expected number of points for this family. We need the following result.
\begin{lem}\label{lem:gcoprime}
Fix a polynomial $h \in \mathcal{S}_d.$ Then, if $h\neq 0$

\[
\left | \left\{g\in \mathcal{S}_{d} : (g,h)=1 \right\}\right| =q^{d+1} \prod_{(P) | (h)} (1-|P|^{-1}).
\]
\end{lem}

We remark that this Lemma follows directly from the proof of Lemma \ref{the-one-before}.

\begin{prop}\label{beta=11} Fix $\alpha \in \mathbb{P}^1(\F_{q^k})$ of degrees $u$ over $\F_q$. Let $\beta \in \mathbb{P}^1(\F_{q^{u}})$. Then, if $u \leq d$,
 \[\frac{|\mathcal{F}_d^v(\alpha,\beta)|}{|\mathcal{F}_d^v|} = \begin{cases} q^{-u} & \deg(\alpha)=u \not = r_i \forall i, \beta \not = \infty,\\\\
                               0 & \deg(\alpha)=u \not = r_i \forall i, \beta = \infty,\\\\
                              \frac{q^{-r_{i_0}}(\pi(r_{i_0})-\ell_{i_0})}{\pi(r_{i_0})}   & \deg(\alpha)=r_{i_0}, \beta \not = \infty,\\\\
 \frac{\ell_{i_0}}{\pi(r_{i_0})} & \deg(\alpha)=r_{i_0}, \beta = \infty.\\                      \end{cases}
\]
If $u>d$ the above quotient is $O(q^{-d})$.
\end{prop}

\begin{proof}
We first consider the size of the whole family. By Lemma \ref{lem:gcoprime} we have
\begin{align} \label{generalv}
|\mathcal{F}_d^v| &=\sum_{\deg P_{i,j}=r_i, \mbox{\small all different}} |\{g \in \mathcal{S}_d\, : \, (g,h)=1\}| \nonumber\\
&=  q^{d+1}\prod_{i=1}^m (1-q^{-r_i})^{\ell_i}\sum_{\deg P_{i,j}=r_i, \mbox{\small all different}}1.
\end{align}
If $\deg(\alpha)=u \not = r_i$, and $\beta \in \F_{q^u}$, then by  Lemma \ref{beta=0} it suffices to find $|\mathcal{F}_d^v(\alpha,\beta)|$ for $\beta=0$.
If this is the case, then we need $g(\alpha)=0$, or $m_\alpha \mid g$.
\begin{align*}
|\mathcal{F}_d^v(\alpha,\beta)| &=\sum_{\deg P_{i,j}=r_i, \mbox{\small all different}} |\{g\in \mathcal{S}_d\, : \, (g,h)=1, m_\alpha \mid g\}|\\
&=  q^{d+1-u}\prod_{i=1}^m (1-q^{-r_i})^{\ell_i}\sum_{\deg P_{i,j}=r_i, \mbox{\small all different}}1.
\end{align*}

If $\deg(\alpha)=u \not = r_i$, and $\beta =\infty,$ we get a contradiction and thus
\[|\mathcal{F}_d^v(\alpha,\infty)|=0.\]

Now assume that $\deg(\alpha)= u=r_{i_0}$, for some $i_0$ and that $\beta \in \F_{q^u}$. By Lemma \ref{beta=0} we can again assume that $\beta=0$. In this case we need to impose the condition that $h(\alpha)\not = 0$. Therefore,
\[|\mathcal{F}_d^v(\alpha,\beta)| =  q^{d+1-r_{i_0}}\prod_{i=1}^m (1-q^{-r_i})^{\ell_i}\sum_{\deg P_{i,j}=r_i, P_{i_0,j}\not = m_\alpha, \mbox{\small all different}}1.\]
Finally, if $\deg(\alpha)= r_{i_0}$ for some $i_0$ and $\beta =\infty$, we need that $h(\alpha)=0$ and $g(\alpha)\not = 0$.
\[|\mathcal{F}_d^v(\alpha,\infty)| =  q^{d+1}\prod_{i=1}^m (1-q^{-r_i})^{\ell_i}\sum_{\deg P_{i,j}=r_i, \exists P_{i_0,j} = m_\alpha, \mbox{\small all different}}1.\]

The result now follows from the identity
\[|\{\deg P_{i,j}=r_i, \mbox{\small all different}\}|= \prod_{i=1}^m\binom{\pi(r_i)}{\ell_i}.\]
\end{proof}

We are now ready to prove the main result of this section.

\begin{thm}\label{thm:vgeneral}
The expected number of $\F_{q^k}$-points on an Artin-Schreier cover with poles given by the factorization type $v$
 defined over $\F_q$ is
\[\begin{cases} q^{k} +1 & p \nmid k, \\ \\q^k+1+(p-1)q^{k/p}+(p-1)\left(1-\sum_{r_i\mid k}\ell_ir_i\right) & p \mid k. \end{cases}\]
 \end{thm}

\begin{proof}
We can assume that $p\nmid r_i$. This is because the $\F_q$-isomorphisms $(x,y)\mapsto (x, y+ax^k)$ allow us to eliminate all the terms in $h$ such that $x$ appears to a power multiple of $p$.

By Lemma \ref{expectednumberfiber}, the final count becomes
\begin{align*}
&\sum_{\alpha \in \mathbb{P}^1(\F_{q^k})} \frac{|\mathcal{F}_d^v(\alpha,\infty)|}{|\mathcal{F}_d^v|}+\sum_{\alpha \in \mathbb{P}^1(\F_{q^k})} \sum_{\beta\in \F_{q^{\deg(\alpha)}}, \tr_k(\beta)=0}p\frac{|\mathcal{F}_d^v(\alpha,\beta)|}{|\mathcal{F}_d^v|}  \\
&=\sum_{r_i\mid k}\frac{\ell_i}{\pi(r_i)}\sum_{\alpha \in \mathbb{P}^1(\F_{q^k}), \deg(\alpha)= r_i}1+ \sum_{\alpha \in \mathbb{P}^1(\F_{q^k})} \sum_{\beta \in \F_{q^{\deg(\alpha)}}, \tr_k(\beta)=0}pq^{-\deg(\alpha)}\\&-\sum_{r_i\mid k}\frac{\ell_i}{\pi(r_i)}\sum_{\alpha \in \mathbb{P}^1(\F_{q^k}), \deg(\alpha)= r_i}\sum_{\beta \in \F_{q^{r_i}}, \tr_k(\beta)=0}pq^{-r_i}.
\end{align*}

If $p\nmid k$, then $\tr_k(\beta)=0$ if and only if $\tr_u(\beta)=0$ and there are $\frac{q^u}{p}$ in $\F_{q^u}$ with that property. Thus we obtain $q^k+1$.
If $p\mid k$, then since $p\nmid r_i$, if $r_i \mid k$ then $p\mid \frac{k}{r_i}$ and $\tr_k(\beta)=0$ for $\beta \in \F_{q^{r_i}}$.
The final count then becomes
\begin{align*}
&\sum_{\alpha \in \mathbb{P}^1(\F_{q^k})} \sum_{\beta \in \F_{q^{\deg(\alpha)}}, \tr_k(\beta)=0}pq^{-\deg(\alpha)}+\sum_{r_i\mid k}\frac{\ell_i}{\pi(r_i)}\sum_{\alpha \in \mathbb{P}^1(\F_{q^{r_i}}), \deg \alpha =r_i}\left(1-\sum_{\beta \in \F_{q^{r_i}}, \tr_k(\beta)=0}pq^{-r_i}\right)\\
&=q^k+1+(p-1)(q^{k/p}+1)-\sum_{r_i\mid k}\frac{\ell_i}{\pi(r_i)}\sum_{\alpha \in \mathbb{P}^1(\F_{q^{r_i}}), \deg \alpha =r_i}(p-1)\\
&=q^k+1+(p-1)q^{k/p}+(p-1)\left(1-\sum_{r_i\mid k}\ell_ir_i\right).\\
\end{align*}
\end{proof}

 Now suppose that we take the $p$-rank 0 family. We recall that this corresponds to $v=(1^d)$. A simple application of Theorem \ref{thm:vgeneral} yields the following.
\begin{thm}\label{thm:rank0}
 The expected number of $\F_{q^k}$-points on a $p$-rank 0 Artin-Schreier cover in $\mathcal{AS}_{\mathfrak{g},0}$  defined over $\F_q$ is
\[\begin{cases} q^{k} +1 & p \nmid k, \\\\ q^{k} +1 +(p-1)q^{k/p} & p \mid k. \end{cases}\]
\end{thm}
This recovers the result from \cite{entin}.

Finally we consider the family of curves with $p$-rank equal to $p-1$. This means that we consider the case when $f(x)$ is a rational function with exactly 2 poles.  If the poles happen to be at $\F_q$-rational points, we are in the case corresponds to $v=(1^{d_1},1^{d_2}).$ Note that in this case we could use an automorphism of $\mathbb P^1(\F_q)$ to move the two poles to zero and infinity, and therefore this case corresponds to the case when $f(X)$ is a Laurent polynomial. Otherwise, the two poles have to be $\F_q$ Galois conjugates points in $\F_{q^2}$ and we found ourselves in the case of prescribed factorization $v=(2^d)$. The final answer for the whole $p$-rank equal to $p-1$ stratum is given by taking the average between these two cases.
Again, by applying Theorem \ref{thm:vgeneral} we get the third part of Theorem \ref{thmnumberofpoints}.

\begin{thm}\label{thm:rank1}
 The expected number of $\F_{q^k}$-points on a $p$-rank $p-1$ Artin-Schreier cover in $\mathcal{AS}_{\mathfrak{g},p-1}$ defined over $\F_q$ is
\[\begin{cases} q^{k} +1 & p \nmid k, \\\\ q^{k} +1 +(p-1)(q^{k/p}-1) & p \mid k,\, k \mbox{ even},\\\\q^{k} +1 +(p-1)q^{k/p} & p \mid k,\, k \mbox{ odd}. \end{cases}\]
\end{thm}
\begin{proof} The different formulas occur when $p\mid k$. 
 For $k$ even we get that both $1 \mid k$ and $2 \mid k$ and therefore we always get $q^{k} +1 +(p-1)(q^{k/p}-1)$ for $p\mid k$. When $k$ is odd, the case $p\mid k$ will yield
\[q^k+1+(p-1)(q^{k/p}-1)\]for $(1^{d_1},1^{d_2})$ 
and 
\[q^k+1+(p-1)(q^{k/p}+1)\]
for $(2^d)$. 

Each case happens half of the time. To see this, notice that $(2^d)$ corresponds to counting degree $2$ irreducible monic polynomials over $\F_q$ while $(1^{d_1},1^{d_2})$ corresponds to counting degree $2$ reducible monic polynomials with two different roots over $\F_q$. The number of degree 2 monic polynomials that are not squares is $q^2-q$, and exactly half of them are reducible. We take the average and obtain the final result.  
\end{proof}

We now proceed to the case where we fix several values, which will be needed for the computation of the moments.

\begin{prop}\label{vmulti} Let $\alpha_1, \dots, \alpha_n \in \PP^1(\F_{q^k})$ of degrees $u_1, \dots, u_n$ over $\F_q$ such that none of the $\alpha_i$ are conjugate to each other. Let $\beta_i \in \F_{q^{u_i}}$ for $1\leq i \leq n$. Then
 \[\frac{|\mathcal{F}_d^v(\alpha_1, \dots, \alpha_n,\beta_1, \dots, \beta_n)|}{|\mathcal{F}_d^v|} = \prod_{i=1}^m (1-\tau(r_i,\ell_i;u_1,\dots,u_n))q^{-(u_1+\cdots+u_n)}+ O(q^{(\varepsilon-1) d}),\]
 where $0\leq \tau(r_i,\ell_i;u_1,\dots,u_n)\leq 1$ is a constant that depends on the number of $u_j$'s that are equal to $r_i$ and is equal to zero if
 $u_j \not = r_i$ for all $j$.
\end{prop}

\begin{proof} Without loss of generality we can assume that $\beta_1, \dots, \beta_\ell$ are not zero and that $\beta_{\ell+1}=\cdots=\beta_n=0$.
 We have that
 \begin{align} \label{generalvvalues}
&|\mathcal{F}_d^v(\alpha_{1}, \dots, \alpha_{n},\beta_{1}, \dots, \beta_{n})| \nonumber \\ &=\sum_{{\deg P_{i,j}=r_i, \mbox{\small all different}}\atop{ P_{i,j}\not = m_{\alpha}}} \left|\left\{g_1 \in \mathcal{S}_{d-\sum_{j=\ell+1}^nu_j}\, : \, (g_1,h)=1, \frac{g_1(\alpha_i)\prod_{j=\ell+1}^n m_{\alpha_j}(\alpha_i)}{h(\alpha_i)  }= \beta_i, \, 1\leq i \leq \ell\right\}\right|.
 \end{align}
 Notice that $\beta_i^{-1}\in \F_{q^{u_i}}^*$ for $1\leq i \leq \ell$.  By Lemma \ref{the-one-before},
 \begin{eqnarray*}&&\left|\left\{g_1 \in \mathcal{S}_{d-\sum_{j=\ell+1}^nu_j}\, : \, (g_1,h)=1, \frac{h(\alpha_i)}{g_1(\alpha_i)\prod_{j=\ell+1}^n m_{\alpha_j}(\alpha_i)}=\beta_i^{-1}\, 1\leq i \leq \ell \right\}\right|\\
&=& q^{d+1-\sum_{i=1}^n u_i} \prod_{(P) | (h)} (1-|P|^{-1}) + O\left(q^{\varepsilon d}\right)\\
 &=& q^{d+1-\sum_{i=1}^n u_i} \prod_{j=1}^m(1-q^{-r_j})^{\ell_j} + O\left(q^{\varepsilon d}\right).\\
 \end{eqnarray*}

 On the other hand, $|\{\deg P_{i,j}=r_i, \mbox{all different},P_{i,j}\not = m_{\alpha}\}|$ is a product of binomials of the form
 \[\binom{\pi(r_i)-s_i}{\ell_i},\]
 where $s_i$ corresponds to the number of $u_j$'s that equal the particular $r_i$.

 This gives that
 \[\frac{|\{\deg P_{i,j}=r_i, \mbox{\small all different},P_{i,j}\not = m_{\alpha}\}|}{|\{\deg P_{i,j}=r_i, \mbox{\small all different}\}|}\]
 is a product of terms of the form
 \[(1-\tau(r_i,\ell_i;u_1,\dots,u_n))=\frac{\binom{\pi(r_i)-s_i}{\ell_i}}{\binom{\pi(r_i)}{\ell_i}}=\frac{(\pi(r_i)-\ell_i)(\pi(r_i)-\ell_i-1)\cdots(\pi(r_i)-\ell_i-s_i+1)}{\pi(r_i)(\pi(r_i)-1)\cdots(\pi(r_i)-s_i+1)}.\]

 By dividing equation \eqref{generalvvalues} by equation \eqref{generalv}, we get
 \[\frac{|\mathcal{F}_d^v(\alpha_1, \dots, \alpha_n,\beta_1, \dots, \beta_n)|}{|\mathcal{F}_d^v|}=q^{-\sum_{i=1}^n u_i} \prod_{i=1}^m (1-\tau(r_i,\ell_i;u_1,\dots,u_n)) + O(q^{(\varepsilon-1) d}),\]
 where the constant satisfies the desired properties.
\end{proof}

\section{Beurling--Selberg functions} \label{trigapprox}

In this section we start the development of the tools needed to prove Theorem \ref{zeroesthm}. By the functional equation, the conjugate of a root of $Z_{C_f}(u)$ is also a root so we can restrict to considering symmetric intervals. Let  $0<\beta<1$ and set $\mathcal I = [-\beta/2, \beta/2] \subset [-1/2,1/2)$.
Our goal is to estimate the quantity
\[N_\mathcal{I}(f, \psi) := \#\left\{1\leq j \leq\frac{2\mathfrak g}{p-1}:\,\theta_j(f, \psi) \in \mathcal{I}\right\}=\sum_{j=1}^{2\mathfrak g/(p-1)} \chi_{\mathcal{I}}(\theta_j(f, \psi)),\]
where $\chi_{\mathcal{I}}$ denotes the characteristic function of ${\mathcal{I}}$. We are going to approximate $\chi_{\mathcal{I}}$ with Beurling--Selberg polynomials $I_K^\pm$.

In what follows, we use the standard notation $e(x) := e^{2 \pi i x}$. Let $K$ be a positive integer,  and let  $h(\theta) = \sum_{{|k|}\leq K} a_k e(k\theta)$ be a trigonometric polynomial.
Then, the coefficients $a_k$ are given by the Fourier transform
$$a_k = \widehat{h}(k) = \int_{-1/2}^{1/2} h(\theta) e (- k \theta) d\theta.$$

Here is a list of a series of useful properties of the Beurling--Selberg polynomials (see \cite{M}, ch 1.2) that will be used in this paper.

\begin{itemize}
 \item[(a)] The $I^\pm_K$ are trigonometric polynomials of degree $\leq K$,
i.e., $$I_K^{\pm}(x) = \sum_{|k| \leq K} \widehat{I}_K^{\pm}(k) e(k x).$$
\item[(b)] The Beurling--Selberg polynomials yield upper and lower bounds for the characteristic function:
\begin{eqnarray*}  I_K^- \leq \chi_{\mathcal{I}}\leq I_K^+ .
\end{eqnarray*}
\item[{(c)}] The integral of Beurling--Selberg polynomials approximates the length of the interval: \[\int_{-1/2}^{1/2} I_K^\pm(x) dx =\int_{-1/2}^{1/2} \chi_{\mathcal{I}}(x) dx \pm \frac{1}{K+1}
    = | \mathcal{I} | \pm \frac{1}{K+1}.
    \]
\item[{(d)}] The $I^\pm_K$ are even (because the interval $\mathcal{I}$
is symmetric about the origin). Therefore the Fourier coefficients
are also even, i.e. $\widehat{I}_K^{\pm}(-k) = \widehat{I}_K^{\pm}(k)$ for
$|k| \leq K$.
\item[{(e)}] The nonzero Fourier coefficients of the Beurling--Selberg polynomials approximate those of the characteristic function:
\[|\widehat{I}_K^\pm (k) - \widehat{\chi}_{
 \mathcal{I}}(k) | \leq \frac{1}{K+1}
\quad \Longrightarrow \quad \widehat{I}^\pm_K(k)=\frac{\sin (\pi k|\mathcal{I}|)}{\pi k} + O
\left( \frac{1}{K+1}\right), \quad k \geq 1.
\]
Therefore we obtain the following bound:
\[|\widehat{I}_K^\pm (k)| \leq \frac{1}{K+1} +\min \left \{|\mathcal{I}|, \frac{\pi}{|k|}\right \}, \quad 0<|k|\leq K.\]
\end{itemize}

We now  list some results that will be useful in future sections.
\begin{prop}(Proposition 4.1, \cite{fr}) \label{propFR} For $K\geq 1$ such that $K|\mathcal{I}|>1$, we have
\begin{eqnarray*}
\sum_{k \geq 1} \widehat{I}_K^\pm (2k)&=&O(1),\\
\sum_{k \geq 1} \widehat{I}_K^\pm (k)^2 k&=&\frac{1}{2\pi^2} \log (K|\mathcal{I}|) +O(1),\\
\sum_{k \geq 1} \widehat{I}_K^+ (k)\widehat{I}_K^- (k) k&=&\frac{1}{2\pi^2} \log (K|\mathcal{I}|) +O(1).\\
\end{eqnarray*}
\end{prop}
We remark that for a given $K$ the above sums are actually finite, since the Beurling--Selberg polynomials $I_K^\pm$ have degree at most $K$.
We will also need the following estimates.
\begin{prop}(Proposition 5.2, \cite{BDFLS}) \label{propmanysums} For $\alpha_1,\dots, \alpha_r, \gamma_1,\dots, \gamma_r>0$, and $\beta_1,\dots,\beta_r \in \mathbb{R}$, we have,
 \[\sum_{k_1,\dots,k_r \geq 1} {\widehat{I}_K^\pm(k_1)}^{\alpha_1} \dots {\widehat{I}_K^\pm(k_r)}^{\alpha_r}k_1^{\beta_1}\dots k_r^{\beta_r}q^{-\gamma_1k_1-\dots-\gamma_r k_r}=O(1).\]
For $\alpha_1,\alpha_2,\gamma >0$, and $\beta \in \mathbb{R}$,
\[\sum_{k\geq1} {\widehat{I}_K^\pm(k)}^{\alpha_1}  {\widehat{I}_K^\pm(2k)}^{\alpha_2} k^\beta q^{-\gamma k}=O(1).\]
\end{prop}

\section{Set-up for the distribution of the zeroes}

We state here an explicit formula that will be used to relate $L(u,f,\psi)$ to the Beurling--Selberg polynomials. Recall that $2\mathfrak g=(p-1)(\Delta-1)$.

\begin{lem}\label{Explicit-Formula}(\cite{BDFLS}, Lemma 3.1)
Let $h(\theta) = \sum_{{|k|}\leq K}\widehat{h}(k)e(k\theta)$  be a trigonometric polynomial. Let $\theta_j(f, \psi)$ be
the eigenangles of the $L$-function $L(u,f,\psi)$.
Then we have
\begin{equation}
\sum_{j=1}^{\Delta-1}h(\thet) = (\Delta-1)\widehat{h}(0) - \sum_{k=1}^{K}\frac{\widehat{h}(k)S_k(f,\psi) + \widehat{h}(-k)S_k(f,\overline{\psi})}{q^{k/2}},
\end{equation}
where
\[
S_k(f, \psi) = \sum_{x\in \mathbb P^1( \mathbb F_{q^k}) \atop f(x) \neq \infty} \psi(\tr_k(f(x))).
\]
\end{lem}

We use the Beurling--Selberg approximation of the characteristic function of the interval $\mathcal I$ to rewrite $N_{\mathcal{I}}(f, \psi)$ and $N_{\mathcal{I}}(C_f)$ where
$f$ varies over one of the families $\mathcal{F}_d$.
By Property (b) of the Beurling--Selberg polynomials, we have
$$
\sum_{j=1}^{\Delta-1} I_K^{-}(\theta_j(f,\psi))  \leq N_\mathcal{I}(f, \psi) \leq \sum_{j=1}^{\Delta-1} I_K^{+}(\theta_j(f,\psi)) ,
$$
and using the explicit formula of Lemma \ref{Explicit-Formula} and Property (c), we have
\begin{eqnarray*}
\sum_{j=1}^{\Delta-1} I_K^{\pm}(\theta_j(f,\psi))  = (\Delta-1) |\mathcal{I}| - S^{\pm}(K, f, \psi) \pm \frac{\Delta-1}{K+1} \\
\end{eqnarray*}
where
\begin{eqnarray} 
S^{\pm}(K, f,\psi) :=  \sum_{k=1}^K \frac{\widehat{I}^\pm_K(k)S_k(f, \psi)+\widehat{I}^\pm_K(-k)S_k(f, \bar{\psi})}{q^{k/2}} . \end{eqnarray}

This gives
\begin{eqnarray}
\label{T-estimate-1}
- S^{-}(K, f,\psi) -\frac{\Delta-1}{K+1}  \leq N_\mathcal{I}(f, \psi) - (\Delta-1)|\mathcal{I}| \leq - S^{+}(K, f,\psi) + \frac{\Delta-1}{K+1},
\end{eqnarray}
and
\begin{eqnarray}
\label{T-estimate-allzeroes-1}
- \sum_{h=1}^{p-1} S^{-}(K, f,\psi^h) - \frac{2\mathfrak g}{K+1}  &\leq& N_\mathcal{I}(C_f) - 2\mathfrak g|\mathcal{I}|  \leq - \sum_{h=1}^{p-1} S^{+}(K, f,\psi^h) + \frac{2\mathfrak g}{K+1}.
\end{eqnarray}

In the next section we are going to compute the moments
\begin{eqnarray*}
\frac{1}{|\mathcal{F}_d|}
\sum_{f \in \mathcal{F}_d} S^\pm(K, f, \psi^h)^n
\quad \mbox{and} \quad
\frac{1}{|\mathcal{F}_d|} \sum_{f \in \mathcal{F}_d} S^\pm(K, C_f)^n
\end{eqnarray*}
where
\begin{eqnarray} \label{def-SCf}
S^\pm(K, C_f)^n =
\sum_{h_1, \dots, h_n=1}^{p-1} S^\pm(K, f, \psi^{h_1}) \dots S^\pm(K, f, \psi^{h_n}).
\end{eqnarray}

We will show that they approach the Gaussian moments when properly normalized. We will then use this result to show that
\[ \frac{N_{\mathcal{I}}(C_f)
- 2\mathfrak{g} |\mathcal{I}|}{\sqrt{\frac{2(p-1)}{\pi^2} \log(\mathfrak{g} |\mathcal{I}|)}}\]
converges to a normal distribution as $\mathfrak{g} \rightarrow \infty$ since
it converges in mean square to \[\frac{S^{\pm}(K, C_f)}{{\sqrt{\frac{2(p-1)}{\pi^2} \log(\mathfrak{g} |\mathcal{I}|)}}}.\]

\section{Moments}

Our goal is to compute the moments of $S^\pm(K, C_f)$ when $f$ varies in any of the families of curves ${\mathcal F}_d^{\ord}$, ${\mathcal F}_d^{\full}$, and ${\mathcal F}_d^{v}$.

\begin{defn} \label{definitionE}
Let
\begin{eqnarray*}
E_{\FF_d}(u) = \begin{cases} (1 + q^{-u} - q^{-2u})^{-1} & \mathcal{F}_d = {\mathcal F}_d^{\ord},\\
(1+q^{-u})^{-1} & \mathcal{F}_d = {\mathcal F}_d^{\full},\\
\displaystyle \frac{\pi(r_i)-\ell_i}{\pi(r_i)}  & \mbox{$\mathcal{F}_d = {\mathcal F}_d^{v}$ and $u=r_i$ for some $i$}, \\
1 & \mbox{$\mathcal{F}_d = {\mathcal F}_d^{v}$ and $u \neq r_i$ for any $i$}.
\end{cases}
\end{eqnarray*}

More generally, we have
\begin{eqnarray*}
E_{\FF_d}(u_1,\dots,u_n) = \begin{cases} \displaystyle\prod_{i=1}^n E_{\FF_d}(u_i) & \mathcal{F}_d = {\mathcal F}_d^{\ord}, {\mathcal F}_d^{\full},\\
\displaystyle \prod_{i=1}^m (1-\tau(r_i,\ell_i;u_1,\dots,u_n)) & \mbox{$\mathcal{F}_d = {\mathcal F}_d^{v}$},\\
\end{cases}
\end{eqnarray*}
where $\tau(r_i,\ell_i;u_1,\dots,u_n)$ is as defined in Proposition \ref{vmulti}.
\end{defn}

\begin{rem}\label{Eestimate} Let $\FF_d$ be any one of the families considered. Then
 \[E_{\FF_d}(u)=1+O\left(uq^{-u}\right).\]
The estimate can be improved to $E_{\FF_d}(u)=1+O\left(q^{-u}\right)$ for ${\mathcal F}_d^{\ord}$ and ${\mathcal F}_d^{\full}$. In the case of ${\mathcal F}_d^{v}$, we are assuming that the $\ell_i$
are fixed constants and using the estimate  $\pi(m)=\frac{q^m}{m}+O\left(\frac{q^{m/2}}{m}\right)$ (see \cite{rosen}, Theorem 2.2).

In addition, we have that
\[E_{\FF_d}(u_1,\dots,u_n) \ll 1.\]
\end{rem}

From now on we will use the notation $\alpha_1\sim \alpha_2$ to indicate that $\alpha_1$ and $\alpha_2$ are Galois conjugate, and $\alpha_1 \not \sim \alpha_2$ for the opposite statement.

Then, for all the families under consideration we have the following result.

\begin{lem} \label{formulasforvalues} Let $\alpha \in \PP^1(\F_{q^k})$ of degree $u$ over $\F_q$. Let $\beta \in \mathbb F_{q^{u}}$.
Let $\mathcal{F}_d$ be any of the families under consideration. Then
\begin{eqnarray} \label{onepoint}
\frac{|\mathcal{F}_d(\alpha,  \beta)|}{|\mathcal{F}_d|} &=&  \frac{|\mathcal{F}_d(\alpha,  0)|}{|\mathcal{F}_d|}
= \frac{E_{\FF_d}(u)}{q^{u}} + O \left( q^{-d/2} \right).
\end{eqnarray}
Let $\alpha_1, \alpha_2 \in \PP^1(\F_{q^k})$ of degrees $u_1, u_2$ respectively over $\mathbb F_q$.  Let $\beta_1 \in \mathbb F_{q^{u_1}}$, $\beta_2 \in \mathbb F_{q^{u_2}}$.
Let $\mathcal{F}_d$ be any of the families under consideration. Then, if $\alpha_1 \not \sim \alpha_2$,
\begin{eqnarray} \label{twopoints}
\frac{|\mathcal{F}_d(\alpha_1, \alpha_2, \beta_1, \beta_2)|}{|\mathcal{F}_d|}
&=& \frac{E_{\FF_d}(u_1, u_2)}{q^{u_1+u_2}} + O\left(q^{-d/2}\right),
\end{eqnarray}
where $E_{\FF_d} (u_1, u_2)$ does not depend on the values of $\beta_1, \beta_2$.

If $\alpha_1 \sim \alpha_2$, and $\beta_1 \sim \beta_2$ by the same automorphism,
\begin{eqnarray} \label{twopointsconj}
\frac{|\mathcal{F}_d(\alpha_1, \alpha_2, \beta_1, \beta_2)|}{|\mathcal{F}_d|}
&=& \frac{|\mathcal{F}_d(\alpha_1, \beta_1)|}{|\mathcal{F}_d|}=  \frac{E_{\FF_d}(u_1)}{q^{u_1}} + O\left(q^{-d/2}\right).
\end{eqnarray}
Otherwise, we get zero.

Let $\alpha_1, \dots, \alpha_n \in \PP^1(\F_{q^k})$ of degrees $u_1, \dots, u_n$ over $\F_q$ and let $\beta_i \in \F_{q^{u_i}}$ for $1\leq i \leq n$.

If none of the $\alpha_i$ are conjugate to each other.  Then
\begin{eqnarray} \label{npoints}
\frac{|\mathcal{F}_d(\alpha_1, \dots, \alpha_n, \beta_1, \dots \beta_n)|}{|\mathcal{F}_d|}
&=& \frac{E_{\FF_d}(u_1, \dots, u_n)}{q^{u_1+ \dots + u_n}} + O(q^{-d/2}),
\end{eqnarray}
where $E_{\FF_d} (u_1, \dots, u_n)$ does not depend on the values of $\beta_1, \dots, \beta_n$.

If some of the $\alpha_i$'s are conjugate to others, then we get zero, unless the corresponding $\beta_i$'s are conjugate by the same automorphisms and in that case we get formula \eqref{npoints},
where the $u_i$'s correspond to the degrees for each of the {\em different} conjugacy classes of the $\alpha_i$'s.

\end{lem}

\begin{proof} This follows from Propositions \ref{nalpha}, \ref{prop:5.6/4.3}, \ref{fullalphas}, \ref{quotientfull}, \ref{beta=11}, and \ref{vmulti}. \end{proof}

We recall that for a family $\mathcal{F}$, a function $G$ depending on $f$, and a vector ${\bm \alpha} = (\alpha_1, \dots, \alpha_n)$, we have the notation
\begin{eqnarray*}
\left< G(f) \right>_{\FF} &:=& \frac{1}{| \mathcal{F}|}
\sum_{{f \in \mathcal{F}}} G(f), \\
\left< G(f) \right>_{\FF, {\bm \alpha}} &:=& \frac{1}{| \mathcal{F}|}
\sum_{{f \in \mathcal{F}}\atop
{f(\alpha_i) \neq \infty, 1 \leq i \leq n}} G(f) .
\end{eqnarray*}

The main idea in the computations of moments
is that if we sum the value of a non-trivial additive character $\psi$ evaluated
at a linear combination of the traces $\tr_{u_i}(\beta_i)$
over all $\beta_i \in \F_{q^{u_i}}$ for $1 \leq i \leq s$,
then the sum will be 0 unless each coefficient is divisible by $p$.
\begin{lem}  \label{keypoint}
Let $m_1, \dots, m_s \in
{\mathbb Z}$, and $\psi$ a non-trivial additive
character of $\F_p$. Then,
\begin{eqnarray*}
\sum_{\beta_i \in \F_{q^{u_i}} \; 1 \leq i \leq s} \psi( m_1 \tr_{u_1}(\beta_1) + \dots + m_s  \tr_{u_s}(\beta_s))
= \begin{cases} q^{u_1+\dots+u_s} & p \mid m_i \; \mbox{for} \; 1 \leq i \leq s, \\\\
0 & \mbox{otherwise}. \end{cases} \end{eqnarray*}
\end{lem}

\subsection{First moment}

\begin{lem} \label{lem:avg1}Let $h$ be an integer such that $p\nmid h$, $e \in \{-1, 1\}$, and  $k>0$. Let $\alpha \in \F_{q^k}$ of degree $u$ over $\F_q$.
Let $\mathcal{F}_d$ be any of the families under consideration. We have,
\[\left< \psi(eh\tr_k f(\alpha))\right>_{\FF_d, \alpha}  = \begin{cases} E_{\FF_d}(u)+O\left(q^{u-d/2}\right) & p\mid \frac{k}{u},\\\\
O\left(q^{u-d/2}\right) & \textrm{otherwise.}
\end{cases} \]
\end{lem}
\begin{proof}
By reversing the order of summation, we obtain
\begin{eqnarray*}
 \left< \psi(eh\tr_k f(\alpha))\right>_{\FF_d, \alpha}  &=& \sum_{\beta \in \F_{q^u}}\psi(e h \tr_k(\beta))\frac{|\FF_d(\alpha,\beta)|}{|\FF_d|}.
\end{eqnarray*}
We now apply Lemma \ref{formulasforvalues} in order to obtain
\[\frac{E_{\FF_d}(u)}{q^u} \sum_{\beta\in \F_{q^u}} \psi\left(\frac{ehk}{u} \tr_u(\beta)\right)+O\left(q^{u-d/2}\right).\]
Lemma \ref{keypoint} implies that the main term  is zero unless $p\mid \frac{k}{u}$. This completes the proof of the statement.
\end{proof}

For positive integers $k, h$ with $p\nmid h$ and $e \in \{-1, 1\}$, set
\begin{eqnarray*}
M^{k,e,h}_{1,d}&:=& \left<q^{-k/2} \sum_{{\alpha \in \F_{q^k}}\atop{f(\alpha)\not =\infty}} \psi(e h \tr_k f(\alpha))\right>_{\FF_d}\\
&=& q^{-k/2} \sum_{{\alpha \in \F_{q^k}}} \left<\psi(e h \tr_k f(\alpha))\right>_{{\FF_d},\alpha}.\\
\end{eqnarray*}

Lemma \ref{lem:avg1} has the following consequence.
\begin{thm} \label{cor:moment1} Let  $h$ be an integer such that $p\nmid h$ and let $\FF_d$ be any of the families under consideration.
 Then \begin{eqnarray*}M^{k,e,h}_{1,d}&=& e_{p,k}\left(E_{\FF_d}\left(k/p\right)q^{-(1/2-1/p)k}+O\left(q^{-(1/2-1/2p)k}\right)\right)+O\left(q^{3k/2-d/2}\right)\\
       &=&O\left(q^{-(1/2-1/p)k}+q^{3k/2-d/2}\right),
      \end{eqnarray*}
where
\[e_{p,k}=\begin{cases} 0 & p\nmid k, \\
1 & p\mid k. \end{cases}\]
\end{thm}
\begin{proof}
By Lemma \ref{lem:avg1}, we have that
\begin{eqnarray*}
 M^{k,e,h}_{1,d} &=& q^{-k/2} \sum_{{u, pu\mid k}\atop{\alpha \in \F_{q^k}, \deg(\alpha)=u}} E_\FFd(u)+q^{-k/2}\sum_{\alpha \in \F_{q^k}}O(q^{\deg(\alpha)-d/2})\\
 &=&\frac{e_{p,k}}{q^{k/2}} \sum_{m, pm\mid k} E_{\FF_d}(m) \pi(m)m +O\left(q^{3k/2-d/2}\right).
\end{eqnarray*}
Finally, if $p\mid k$, the estimates from Remark \ref{Eestimate} yield
\[\sum_{m, pm\mid k}  E_{\FF_d}(m) \pi(m)m =E_{\FF_d}\left(k/p\right)q^{k/p} +O\left(q^{k/2p}\right). \]

\end{proof}

Notice that changing $h$ allows us to vary the character from $\psi$ to $\psi^h$. This will be useful later.

\begin{thm} \label{boundforSpm}
Let $h$ be an integer such that $p\nmid h$ and  let $\mathcal{F}_d$ be any of the families under consideration. Then for any $K$ with $\max \{1, 1/|\mathcal{I}|\}<K<d/3$,
 \[\left < S^\pm (K,f,\psi^h)\right>_{\FF_d} =O(1). \]
\end{thm}

\begin{proof}
 We have that
\begin{eqnarray*}
\left<  S^{\pm}(K, f, \psi^h) \right>_{\FF_d} &=&  \sum_{k=1}^K \frac{\widehat{I}^\pm_K(k)\left<S_k(f,\psi^h) \right>_{\FF_d}+\widehat{I}^\pm_K(-k)\left<S_k(f,\bar{\psi}^h)\right>_{\FF_d}}{q^{k/2}}\\
&=&  \sum_{k=1}^K \widehat{I}^\pm_K(k)M^{k,1,h}_{1,d}+\widehat{I}^\pm_K(-k)M^{k,-1,h}_{1,d}\\
&= &  \sum_{k=1}^K \widehat{I}^\pm_K(k)O\left(q^{-(1/2-1/p)k}+q^{3k/2-d/2}\right).\\
\end{eqnarray*}
and the result follows from Proposition \ref{propmanysums}.
\end{proof}

\begin{thm} Let $\FF_d$ be any of the families under consideration. Then,
\begin{eqnarray*}
\left< N_{\mathcal I}(f,\psi) \right>_\FFd = \frac{1}{|\mathcal{F}_d|}
\sum_{f \in \mathcal{F}_d}  N_{\mathcal I}(f,\psi) &=& (\Delta-1)|\mathcal{I}| + O \left( 1 \right) \\
\left< N_{\mathcal I}(C_f) \right>_\FFd  = \frac{1}{|\mathcal{F}_d|}
\sum_{f \in \mathcal{F}_d}  N_{\mathcal I}(C_f) &=& 2\mathfrak g|\mathcal{I}| + O \left( 1 \right).
\end{eqnarray*}
\end{thm}
\begin{proof} This follows from Theorem \ref{boundforSpm} and equations \eqref{T-estimate-1} and \eqref{T-estimate-allzeroes-1}
using $K=\varepsilon d$ for any $0<\varepsilon<1/3$.
\end{proof}

\subsection{Second moment} \label{2mom}

\begin{lem} \label{lemmaindependence}
Let $h_1, h_2$ be integers such that $p\nmid h_1 h_2$, $e_1, e_2 \in \{-1,1 \}$ and  $k_1,k_2 > 0$. Let $\alpha_1 \in \F_{q^{k_1}}$, $\alpha_2 \in \F_{q^{k_2}}$
of degrees $u_1, u_2$ respectively over $\F_q$. For any of the families under consideration, we have,
\begin{eqnarray*}
&&\left< \psi(e_1 h_1\tr_{k_1} f(\alpha_1)+e_2h_2\tr_{k_2} f(\alpha_2)) \right>_{\FF_d, (\alpha_1, \alpha_2)}\\ && \quad \quad =\begin{cases}
E_{\FF_d}(u_1) + \displaystyle O \left( {q^{u_1 -d/2}} \right)& \textrm{$\alpha_1 \sim \alpha_2$, $p \mid
\displaystyle \frac{{e_1 h_1 k_1}+{e_2 h_2 k_2}}{u_1}$}, \\
 \displaystyle O \left(  1 + {q^{u_1+u_2-d/2}} \right)
& \textrm{$\displaystyle \alpha_1 \not\sim \alpha_2$, $p \mid \left( \frac{k_1}{u_1}, \frac{k_2}{u_2} \right)$},
\\ \displaystyle \displaystyle O \left( {{q^{u_1+u_2-d/2}}} \right) & \textrm{otherwise.} \end{cases}
\end{eqnarray*}
\end{lem}
\begin{proof}

Reversing the order of summation, we write
\begin{eqnarray} \nonumber
&&\left< \psi(e_1 h_1\tr_{k_1} f(\alpha_1)+e_2h_2\tr_{k_2} f(\alpha_2)) \right>_{\FF_d, (\alpha_1, \alpha_2)}\\
\label{afterrevsums}
&& \quad \quad = \sum_{\beta_1 \in \F_{q^{u_1}},\beta_2 \in \F_{q^{u_2}}}
\psi(e_1 h_1\tr_{k_1} \beta_1 +e_2h_2\tr_{k_2} \beta_2) \frac{|\mathcal{F}_d (\alpha_1, \alpha_2, \beta_1, \beta_2)|}{|\mathcal F_d|}.
\end{eqnarray}

Assume that $\alpha_1 \not \sim \alpha_2$. By Lemma \ref{formulasforvalues} we can write \eqref{afterrevsums} as
\begin{eqnarray*}
&& \frac{E_{\FF_d}(u_1,u_2)}{q^{u_1+u_2}}
\sum_{{\beta_1 \in \F_{q^{u_1}}}, {\beta_2 \in \F_{q^{u_2}}}}
  \psi \left( \frac{e_1 h_1 k_1}{u_1} \tr_{u_1}\beta_1 +  \frac{e_2 h_2 k_2}{u_2} \tr_{u_2}\beta_2 \right) + O \left({q^{u_1+u_2-d/2}} \right).
\end{eqnarray*}
Then Lemma \ref{keypoint} implies that the sum is zero unless $p \mid \frac{k_1}{u_1}$ and  $p \mid \frac{k_2}{u_2}$.

Now assume that $\alpha_1 \sim \alpha_2$. Then $f(\alpha_1) \sim f(\alpha_2)$ and $\tr_{u_1}f(\alpha_1)  = \tr_{u_1}f(\alpha_2)$. By Lemma \ref{formulasforvalues} we can write \eqref{afterrevsums} as
\begin{eqnarray*}
&& \frac{E_{\FF_d}(u_1)}{q^{u_1}}
\sum_{{\beta_1 \in \F_{q^{u_1}}}}
  \psi \left( \frac{e_1 h_1 k_1+e_2 h_2 k_2}{u_1} \tr_{u_1}\beta_1 \right) + O \left({q^{u_1-d/2}} \right).
\end{eqnarray*}
Then Lemma \ref{keypoint} implies that the sum is zero unless $p \mid  \frac{e_1 h_1 k_1+e_2 h_2 k_2}{u_1}$.

\end{proof}

\begin{lem} \label{entintrick} Let $h_1, h_2$ be integers such that $p\nmid h_1 h_2$, $e_1, e_2 \in \{-1,1 \}$
and  $k_1,k_2 > 0$, $k_1 \geq k_2$.
Let $\FF_d$ be any of the families under consideration. Then, 
\begin{eqnarray*}
\sum_{{{m\mid ({k_1},{k_2})}\atop{mp\nmid {k_1},{k_2}}}\atop{mp \mid (e_1h_1{k_1}+e_2h_2{k_2})}}
 E_{\FF_d} (m) {\pi(m) m^2}
= \begin{cases} E_{\FF_d}(k_1) k_1 q^{k_1} + O \left( {k_1} q^{{k_1}/2}  \right) & k_1=k_2, p \mid(e_1 h_1 + e_2 h_2), \\
0 & k_1=k_2, p \nmid (e_1 h_1 + e_2 h_2), \\
O \left( k_1 q^{k_1/2} \right) & k_1 = 2k_2, \\
O \left( k_1 q^{k_1/3} \right) & k_1 \neq k_2, 2k_2. \end{cases}
\end{eqnarray*}
\end{lem}
\begin{proof}

For the first case
when ${k_1}={k_2}$,  the conditions on the summation indices become $m\mid {k_1}$, $mp\nmid {k_1}$, and $mp\mid (e_1h_1+e_2h_2){k_1}$, a contradiction unless $p\mid (e_1h_1+e_2h_2)$. 
In this case, one gets
\[\sum_{{m\mid {k_1}}\atop{mp\nmid {k_1}}} E_{\FF_d}(m){\pi(m) m^2}    = {E_{\FF_d}(k_1) {k_1}q^{k_1}} +O\left({k_1}q^{{k_1}/2} \right),\]
where we have used the estimates for $\pi(m)$ and $E_{\FF_d}(m)$ discussed in Remark \ref{Eestimate}.

On the other hand, when ${k_1}=2{k_2}$, one gets
\[\sum_{{{m\mid {k_2}}\atop{mp\nmid {k_2}}}\atop{mp \mid (2e_1h_1+e_2h_2){k_2}}}
E_{\FF_d}(m){\pi(m) m^2}
=O \left( {k_1}q^{k_1/2}  \right)
.\]
Finally, if ${k_1}>{k_2}$ but ${k_1}\not = 2{k_2}$, we have $({k_1},{k_2})\leq {k_1}/3$ and
\[\sum_{{{m\mid ({k_1},{k_2})}\atop{mp\nmid {k_1},{k_2}}}\atop{mp \mid (e_1h_1{k_1}+e_2h_2{k_2})}}
E_{\FF_d}(m) {\pi(m) m^2}
=O\left({k_1} q^{{k_1}/3}
\right).\]
This completes the proof.
\end{proof}

For positive integers $k_1,k_2,h_1,h_2$ with $p\nmid h_1h_2$ and $e_1, e_2 \in \{ -1, 1\}$, let
\begin{eqnarray*}
M_{2,d}^{(k_1,k_2),(e_1,e_2),(h_1,h_2)} &:=& \left<  q^{-(k_1+k_2)/2} \sum_{{\alpha_1 \in \F_{q^{k_1}}, \alpha_2 \in \F_{q^{k_2}}} \atop {f(\alpha_1) \neq \infty, f(\alpha_2) \neq \infty}}
\psi(e_1h_1\tr_{k_1} f(\alpha_1)+e_2h_2 \tr_{k_2} f(\alpha_2)) \right>_{\FF_d} \\
&=& q^{-(k_1+k_2)/2} \sum_{{\alpha_1 \in \F_{q^{k_1}}}\atop{\alpha_2 \in \F_{q^{k_2}}}}
\left< \psi(e_1h_1\tr_{k_1} f(\alpha_1) +e_2h_2 \tr_{k_2} f(\alpha_2)) \right>_{\FF_d, (\alpha_1, \alpha_2)}.
\end{eqnarray*}

Using Lemma \ref{entintrick}, we can prove
the following analogue of Theorem 8 in \cite{entin}.

\begin{thm}\label{Mcovariance} Let $0<h_1, h_2 \leq (p-1)/2$, $e_1, e_2 \in \left\{ -1, 1 \right\}$,
$k_1 \geq k_2 > 0$, and let $\FF_d$ be any of the families under
consideration.
Then
\begin{eqnarray*}
M_{2,d}^{({k_1},{k_2}),(e_1,e_2),(h_1,h_2)} &=&
\begin{cases}
 \delta_{{k_1},{k_2}}  \left(E_{\FF_d}(k_1) k_1  + O\left(k_1 q^{-{k_1}/2} + k_1 q^{(k_1-d/2)} \right)\right) & e_1 = -e_2, \,h_1=h_2,\\
 0 & \text{ otherwise,}
\end{cases}\\
&& + \delta_{{k_1},2{k_2}} O \left(k_1 q^{-k_2/2} + {k_1} q^{{k_2}/2-d/2}\right) \\
&&+ O \left({k_1} q^{-{k_2}/2-{k_1}/6} + k_1 q^{k_1/6 - k_2/2 - d/2} \right)\\
&& +O \left(
 q^{(1/p - 1/2)({k_1}+{k_2})} + q^{3({k_1}+{k_2})/2 - d/2}\right)
\end{eqnarray*}
where
\begin{eqnarray*}
\delta_{{k_1},{k_2}} = \begin{cases} 1, & {k_1}={k_2}, \\0, & {k_1} \neq {k_2}. \end{cases}
\end{eqnarray*}
\end{thm}

\begin{proof}

From Lemma \ref{lemmaindependence}, we have
\begin{eqnarray*}
M_{2,d}^{({k_1},{k_2}),(e_1,e_2),(h_1,h_2)}  &=&
\frac{e_{p,e_1h_1{k_1}+e_2h_2{k_2}}}{q^{({k_1}+{k_2})/2}}
\sum_{{{m\mid ({k_1},{k_2})}\atop{mp\nmid {k_1},{k_2}}}\atop{mp \mid (e_1h_1{k_1}+e_2 h_2{k_2})}} {\pi(m)m^2}
\left( E_{\FF_d}(m) + O ( q^{m-d/2} ) \right) \\
&& + O \left( \frac{e_{p,{k_1}} e_{p,{k_2}}}{q^{({k_1}+{k_2})/2}}
\sum_{{\deg{\alpha_1}=u_1, \deg{\alpha_2}=u_2} \atop {p \mid \frac{k_1}{u_1}, p \mid \frac{k_2}{u_2}}}
\left( 1 + q^{u_1+u_2 - d/2} \right) \right)\\
&& + O \left( \frac{1}{q^{({k_1}+{k_2})/2}} \sum_{{\deg{\alpha_1}=u_1, \deg{\alpha_2}=u_2} \atop {u_1 \mid k_1, u_2 \mid k_2}} q^{u_1+u_2 - d/2} \right).
\end{eqnarray*}
It is easy to see that the last two terms are
$$ O \left(
 q^{(1/p - 1/2)({k_1}+{k_2})} + q^{3({k_1}+{k_2})/2 - d/2}\right).
$$

For the first term, we use Lemma \ref{entintrick}. As a final observation, the condition $p\mid e_1h_1+e_2h_2$ translates into
$h_1=h_2$ and $e_1=-e_2$ because of the restriction on the possible values for $h_1, h_2$.
This concludes the proof of the theorem.
\end{proof}

Using Lemma \ref{entintrick}, we can prove the following result which will also be used in
the general moments.

\begin{prop} \label{recordresult} Let $h_1, h_2$ be integers such that $p\nmid h_1 h_2$, $e_1, e_2 \in \{-1,1 \}$
and  $k_1,k_2 > 0$.
Let $\FF_d$ be any of the families under consideration. Then,
\begin{eqnarray*}
&&\sum_{k_1, k_2 = 1}^K \widehat{I}_K^{\pm}(e_1 k_1) \widehat{I}_K^{\pm}(e_2 k_2) q^{-(k_1+k_2)/2} \sum_{{{m\mid ({k_1},{k_2})}\atop{mp\nmid {k_1},{k_2}}}\atop{mp \mid (e_1h_1{k_1}+e_2 h_2{k_2})}} E_{\FF_d}(m) \pi(m)m^2 \\
&& \quad \quad \quad =
\begin{cases} \displaystyle \frac{1}{2\pi^2} \log{\left( K |\mathcal{I}| \right)} + O(1) &
p \mid (e_1 h_1 + e_2 h_2), \\\\
O(1) & \mbox{otherwise}. \end{cases}
\end{eqnarray*}
\end{prop}

\begin{proof} 
Using Lemma \ref{entintrick}, we have  the sum is
\begin{eqnarray*}
&&e_{p, e_1 h_1 + e_2 h_2} \sum_{k_1=1}^K \widehat{I}_K^{\pm}(k_1) \widehat{I}_K^{\pm}(- k_1) \left( E_{\FF_d}(k_1) k_1 + O \left( k_1 q^{-k_1/2} \right) \right)
\\
&& + O \left( \sum_{k_1=1}^K k_1 q^{-k_1/4} + \sum_{k_1, k_2=1}^K k_1 q^{-k_1/6} q^{-k_2/2} \right)
\\
&&=e_{p, e_1 h_1 + e_2 h_2} \sum_{k_1=1}^K \widehat{I}_K^{\pm}(k_1) \widehat{I}_K^{\pm}(- k_1) E_{\FF_d}(k_1)
k_1 + O ( 1 ).
\end{eqnarray*}
Now the estimates from Remark \ref{Eestimate} and Proposition \ref{propFR} yield
\begin{eqnarray*}
\sum_{k_1=1}^K \widehat{I}_K^{\pm}(k_1) \widehat{I}_K^{\pm}(- k_1)  E_{\FF_d}(k_1)  k_1
&=&  \sum_{k_1=1}^K \widehat{I}_K^{\pm}(k_1) \widehat{I}_K^{\pm}(- k_1) k_1 +
O \left( \sum_{{k_1=1}}^K  k_1^2 q^{-k_1} \right)\\
 &=& \frac{1}{2\pi^2} \log( K |\mathcal{I}| ) + O ( 1 ),
\end{eqnarray*}
which finishes the proof of the statement.
\end{proof}

Finally, we are able to compute the covariances.

\begin{thm}\label{covariance}Let  $0< h_1, h_2\leq (p-1)/2$, and let $\FF_d$ be any of the families under consideration.
Then for any  $K$ with $1/|\mathcal I|<K<d/6$,
\[
 \left< S^\pm(K, f, \psi^{h_1}) S^\pm(K, f, \psi^{h_2})\right>_{\FF_d}=\left< S^\pm(K, f, \psi^{h_1}) S^\mp(K, f, \psi^{h_2}) \right>_{\FF_d} =\begin{cases}

\displaystyle \frac{1}{\pi^2} \log (K |\mathcal{I}|)+ O\left(1\right) & h_1=h_2,
 \\
 &\\
  O\left(1\right) & h_1\neq h_2.
  \end{cases}
\]

\end{thm}
\begin{proof}
 By definition,
\begin{eqnarray*}
 &&\left< S^\pm(K, f, \psi^{h_1}) S^\pm(K, f, \psi^{h_2}) \right>_{\FF_d}   \\
&&= \sum_{{k_1},{k_2} =1}^K \widehat{I}_K^\pm({k_1}) \widehat{I}_K^\pm({k_2}) M_{2,d}^{({k_1},{k_2}),(1,1),(h_1,h_2)}  +
\widehat{I}_K^\pm({k_1}) \widehat{I}_K^\pm(-{k_2}) M_{2,d}^{({k_1},{k_2}),(1,-1),(h_1,h_2)}  \\
&& \;\;\;\; + \widehat{I}_K^\pm(-{k_1}) \widehat{I}_K^\pm({k_2}) M_{2,d}^{({k_1},{k_2}),(-1,1),(h_1,h_2)}+\widehat{I}_K^\pm(-{k_1}) \widehat{I}_K^\pm(-{k_2}) M_{2,d}^{({k_1},{k_2}),(-1,-1),(h_1,h_2)}.
\end{eqnarray*}
Using Theorem \ref{Mcovariance} to replace the terms above, we first remark that
the contribution of the last two error terms from Theorem \ref{Mcovariance} to the sum is
$$
\ll  \sum_{{k_1},{k_2} =1}^K
 k_1q^{-{k_2}/2-{k_1}/6} +  k_1q^{k_1/6 - k_2/2 - d/2}
+q^{(1/p - 1/2)({k_1}+{k_2})} + q^{3({k_1}+{k_2})/2 - d/2} \ll 1
$$
provided that $d > 6K$.

Similarly, the contribution of the error terms for $k_1=k_2$ and
$k_1 = 2k_1$ is bounded by
$$\ll \sum_{k=1}^K kq^{-k/2} +k q^{k-d/2}  \ll 1
$$
provided that $d > 2K$.
Finally, the main term comes from summing $E_{\FF_d}(k_1) k_1$ when $k_1=k_2$, and this occurs only
when $h_1 = h_2$ and $\{e_1, e_2\} = \{1, -1\}$. Proceeding as in the proof of Proposition \ref{recordresult},
we then get that
\begin{eqnarray*}
\left< S^\pm(K, f, \psi^{h_1}) ^2 \right>_{\FF_d} &=& 2 \sum_{{k_1}=1}^K \widehat{I}_K^\pm({k_1})\widehat{I}_K^\pm({-k_1}) {k_1} E_{\FF_d}(k_1)
+ O(1)\\
&=& \frac{1}{\pi^2} \log(K |\mathcal{I}|) + O(1).
\end{eqnarray*}

The proof for  $\left< S^\pm(K, f, \psi^{h_1}) S^\mp(K, f, \psi^{h_2}) \right>_{\FF_d}$ follows exactly along the same lines.
\end{proof}

\begin{cor}\label{cor:2ndmoment}
For any  $K$ with $1/|\mathcal I|<K<d/6$,
\[
\left\langle S^\pm(K,C_f)^2 \right\rangle_{\FF_d}=\left\langle S^+(K,C_f)S^-(K,C_f)\right\rangle_{\FF_d}=\frac{2(p-1)}{\pi^2}\log(K|\mathcal I|) + O(1).
\]
\end{cor}
\begin{proof}
First we note that
\begin{eqnarray*}
\left\langle S^\pm(K,C_f)^2 \right\rangle_{\FF_d}=  \sum_{h_1,h_2=1}^{p-1} \left\langle S^{\pm}(K, f, \psi^{h_1}) S^{\pm}(K, f, \psi^{h_2})\right\rangle_{\FF_d}.
\end{eqnarray*}
Notice that by Theorem \ref{covariance}, the mixed average contributes $\frac{1}{\pi^2}\log(K|\mathcal I|)+O(1)$ for each term where $h_1=h_2$ or $h_1=p-h_2$. The proof for $\left\langle S^+(K,C_f)S^-(K,C_f)\right\rangle_{\FF_d}$ is identical.
\end{proof}

\subsection{General moments}\label{genmom}

Let $n, k_1, \dots, k_n$ be positive integers, let $e_1, \dots, e_n$ take values $\pm 1$ and let \\$h_1, \dots, h_n$ be integers such that $p\nmid h_i$, $1\leq i\leq n$.
Let ${\bf k} = (k_1, \dots, k_n)$, ${\bf e} = (e_1, \dots, e_n)$, and ${\bf h}=(h_1,\dots, h_n)$.
Let $\alpha_i \in \F_{q^{k_i}}$, $1 \leq i \leq n$, and let ${\bm \alpha}=(\alpha_1,\dots,\alpha_n)$.
Let $\FF_d$ be any of the families under consideration. Then,
we define
\begin{eqnarray*}
m_{n}^{{\bf k}, {\bf e}, {\bf h}}({\bm \alpha}) &=& \left< \psi(e_1 h_1 \tr_{k_1}f(\alpha_1) + \dots +
e_nh_n \tr_{k_{n}} f(\alpha_n)) \right>_{\FF_d, {\bm \alpha}}\\
&=& \frac{1}{|\FF_d|} \sum_{{f \in \FF_d} \atop {f(\alpha_i) \neq \infty, 1 \leq i \leq n}} \psi(e_1 h_1 \tr_{k_1}f(\alpha_1) + \dots +
e_nh_n \tr_{k_{n}} f(\alpha_n)),
\end{eqnarray*}
and
$$M_{n}^{{\bf k}, {\bf e}, {\bf h}} =
 \sum_{{\alpha_i \in \F_{q^{k_i}}} \atop {i=1, \dots, n}}
q^{-(k_1 + \dots + k_n)/2} m_{n}^{{\bf k}, {\bf e}, {\bf h}}({\bm \alpha}).$$

\begin{lem} \label{generalcase} Let $\FF_d$ be any of the families under consideration.
Let $C_1, \dots, C_s$ be the distinct conjugacy classes of
the $\alpha_1, \dots, \alpha_n$. Let $u_i$ be the degree of the elements of $C_i$.
For $i=1, \dots, s$, let
$$\eta_i = \frac{1}{u_i} \sum_{\alpha_j \in C_i} {e_j} h_j k_j .$$
Then
$$m_{n}^{{\bf k}, {\bf e}, {\bf h}}({\bm \alpha}) = \begin{cases} \displaystyle
  E_{\FF_d}(u_1, \dots, u_s) +O \left( q^{u_1 + \dots + u_s - d/2} \right) & \mbox{if $p \mid \eta_i$ for $1 \leq i \leq s$}, \\
 \\
O \left( q^{u_1 + \dots + u_s - d/2} \right) & \mbox{otherwise}. \end{cases}$$
\end{lem}

\begin{proof}
Renumbering, suppose that $\alpha_i \in C_i$ for $1 \leq i \leq s$.
Since $\tr_{k_i}f(\alpha_i)=\frac{k_i}{u_i} \tr_{u_i} f(\alpha_i)$ for $i=1, \dots, s$,
by the definition of  $\eta_i$, we have that
\begin{eqnarray*}
m_{n}^{{\bf k}, {\bf e}, {\bf h}}({\bm \alpha}) &=& \frac{1}{|\FF_d|}
\sum_{{f \in \mathcal{F}_d} \atop {f(\alpha_i) \neq \infty, 1 \leq i \leq n}} \psi \left( e_1 h_1 \tr_{k_1} f(\alpha_1) + \dots + e_nh_n \tr_{k_n}f(\alpha_n) \right) \\
& =& \frac{1}{|\FF_d|}
\sum_{{f \in \mathcal{F}_d} \atop {f(\alpha_i) \neq \infty, 1 \leq i \leq n}} \psi \left( \eta_1 \tr_{u_1} f(\alpha_1) + \dots + \eta_s \tr_{u_s}f(\alpha_s) \right) \\
&=& \sum_{\beta_i \in \F_{q^{u_i}}, \; 1 \leq i \leq s}
 \psi \left( \eta_1 \tr_{u_1} \beta_1 + \dots + \eta_s \tr_{u_s} \beta_s \right)
 \frac{| \FF_d(\alpha_1, \dots, \alpha_s, \beta_1, \dots, \beta_s)|}{| \FF_d|} \\
 &=& \frac{E_{\FF_d}(u_1, \dots, u_s)}{q^{u_1+\dots+u_s}}
 \sum_{\beta_i \in \F_{q^{u_i}}, \; 1 \leq i \leq s}
 \psi \left( \eta_1 \tr_{u_1} \beta_1 + \dots + \eta_s \tr_{u_s} \beta_s \right)
 + O \left( q^{u_1+\dots+u_s -d/2} \right)
\end{eqnarray*}
by Lemma \ref{formulasforvalues}. The result now follows from Lemma \ref{keypoint}.
\end{proof}

\begin{lem} \label{generalcase2}
$M_{n}^{{\bf k}, {\bf e}, {\bf h}}$ is bounded  by a sum  of terms
$q^{-(k_1 + \dots + k_n)/2} T(k_1, \dots, k_n)$, where each  $T(k_1, \dots, k_n)$ is
a product of elementary terms of the
type
\[ \sum_{{{m\mid (j_1,\dots,j_r)}\atop{mp\mid \sum_{i=1}^r e_i h_i j_i}}}  \pi(m) m^r \]
such that the indices $j_1, \dots, j_r$ of the elementary terms appearing in each $T(k_1, \dots, k_n)$ are in bijection with $k_1, \dots, k_n$.

For $n=2\ell$ even, let $N_{n}^{{\bf k}, {\bf e}, {\bf h}}$ be the sum of all possible terms
$q^{-(k_1 + \dots + k_n)/2} T(k_1, \dots, k_n)$ where the $T(k_1, \dots, k_n)$  are made exclusively of the following nested sums
\begin{equation} 
\sum_{{m_1\mid (j_1,j_{\ell+1})}\atop{m_1p\mid e_1h_1 j_{\ell+1}+ e_{\ell+1}h_{\ell+1} j_{\ell+1}}} \pi(m_1) m_1^2 \dots \sum_{{m_\ell\mid (j_\ell,j_{2\ell})}\atop{m_\ell p\mid e_\ell h_\ell j_{2\ell}+ e_{2\ell}h_{2\ell} j_{2\ell}}} \pi(m_\ell) m_\ell^2 E_{\mathcal{F}_d}(m_1,\dots,m_\ell).
\end{equation}

If $n=2\ell+1$ is odd, let $N_{n}^{{\bf k}, {\bf e}, {\bf h}}$ be the sum of all possible terms $q^{-(k_1 + \dots + k_n)/2} T(k_1, \dots, k_n)$ where $T(k_1, \dots, k_n)$ are made exclusively of
the following nested sums
\begin{eqnarray*} 
&&\sum_{{m_1\mid (j_1,j_{\ell+1})}\atop{m_1p\mid e_1h_1 j_{\ell+1}+ e_{\ell+1}h_{\ell+1} j_{\ell+1}}} \pi(m_1) m_1^2 \dots \sum_{{m_\ell\mid (j_\ell,j_{2\ell})}\atop{m_\ell p\mid e_\ell h_\ell j_{2\ell}+ e_{2\ell}h_{2\ell} j_{2\ell}}} \pi(m_\ell) m_\ell^2 \sum_{{m_{\ell+1}\mid j_{2\ell+1}}\atop{m_{\ell + 1}p\mid e_{2\ell+1}h_{2\ell+1} j_{2\ell+1}}} \pi(m_{\ell+1}) m_{\ell+1}\\
&&\times E_{\mathcal{F}_d}(m_1,\dots,m_\ell, m_{\ell+1}).
\end{eqnarray*}

Let $L_{n}^{{\bf k}, {\bf e}, {\bf h}}$ be the sum of all the  other  terms $q^{-(k_1 + \dots + k_n)/2} T(k_1, \dots, k_n)$ as defined above.
Then, $$M_{n}^{{\bf k}, {\bf e}, {\bf h}} = N_{n, d}^{{\bf k}, {\bf e}, {\bf h}} + O \left( L_{n}^{{\bf k}, {\bf e}, {\bf h}} \right) + O \left( q^{3(k_1+\dots+k_n)/2 - d/2} \right).$$
\end{lem}

\begin{proof} Using Lemma \ref{generalcase}, we first write
$$
M_{n}^{{\bf k}, {\bf e}, {\bf h}} = q^{-(k_1+\cdots +k_n)/2}
\sum_{{{\alpha_i \in \F_{q^{k_i}},} \; {i=1, \dots, n}} \atop {(\alpha_1, \dots, \alpha_n) \in \mathcal{A}}}
 E_{\FF_d}(u_1, \dots, u_s) +
O \left(
q^{3(k_1 + \dots + k_n)/2 - d/2} \right),$$
where the set $\mathcal{A}$ of admissible $(\alpha_1, \dots, \alpha_n)$ are those
where $p \mid \eta_i, \;i=1, \dots, s$. To count the number of admissible
$(\alpha_1, \dots, \alpha_n)$, we first fix a partition of $\{ 1, \dots, n \}$ in
$s$ classes $C_1, \dots C_s$.  Let $k(C_w)$ be the gcd of the $k_i$ such that $i \in C_w$ and let $\delta(C_w)=\sum_{i\in C_w} e_i h_i k_i$.
Then, for any such partition, the number of $(\alpha_1, \dots, \alpha_n) \in \F_{q^{k_1}}
\times \dots \times \F_{q^{k_n}}$ such that $\alpha_i$ and $\alpha_j$ are conjugate when
$i,j$ are in the same class $C_w$ and
which are counted in $\mathcal{A}$ is bounded by
\begin{eqnarray} \label{firstcount}
\prod_{i=1}^s \sum_{{{m\mid k(C_i)}\atop{mp \mid \delta(C_i)}}}  \pi(m)m^{|C_i|},
\end{eqnarray}
where we have used the fact that the number of $(\alpha_1, \dots, \alpha_t) \in
\F_{q^{k_1}} \times \dots \times \F_{q^{k_t}}$ which are conjugate over $\F_q$
is given by
$$\sum_{m \mid (k_1, \dots, k_t)} \pi(m) m^t.$$
Since $E_\FF(u_1, \dots, u_s) \ll 1$ by Remark \ref{Eestimate}, we get the first result of
the statement by
summing \eqref{firstcount} over all partitions of $\{ 1, \dots, n \}$ in
$s$ classes $C_1, \dots C_s$.

Suppose that $n=2\ell$ is even. Then, using inclusion-exclusion, the number of $(\alpha_1, \dots, \alpha_n)
\in \F_{q^{k_1}}\times \cdots \times \F_{q^{k_n}}$
such that $\alpha_i$ and $\alpha_{j}$ are conjugate, if and only if  $i \equiv j (\bmod \ell)$  can be written as
$$
\left( \sum_{{{m_1\mid (k_1,k_{\ell+1})\atop{m_1p\mid e_1 h_1 k_1 + e_{\ell+1} h_{\ell+1}k_{\ell+1}}}}} \pi(m_1) m_1^2 \dots
\sum_{{{m_\ell\mid (k_{\ell},k_{2 \ell})}\atop{m_\ell p\mid e_\ell h_\ell k_\ell + e_{2 \ell} h_{2\ell}k_{e \ell}}}} \pi(m_\ell) m_\ell^2 E_{\FF_d}(m_1, \dots, m_\ell)\right) + S(k_1, \dots, k_n)
$$
where $S(k_1, \dots, k_n)$ is a sum of terms in  $L_{n}^{{\bf k}, {\bf e}, {\bf h}}$.
(We have to do inclusion-exclusion to remove the cases where conjugate values of
$\alpha$ belong to two different classes $C_w$.)

The case of $n=2\ell+1$ follows similarly, taking into account
that one has to multiply by the factor $\displaystyle q^{-k_n/2} \sum_{{{m\mid k_n}\atop{mp\mid e  k_n}}}
\pi(m) m$.
\end{proof}

\begin{thm}\label{thm:Smoments}
Let $\mathcal{F}_d$ be any of the families under consideration. For any $K$ with $1/|\mathcal I| <K< d/n$
\[\left< S^\pm(K, f, \psi)^{n} \right>_\FFd= \left\{\begin{array}{ll}\frac{(2 \ell)!}{\ell! (2\pi^2)^\ell}  \log^\ell(K|\mathcal{I}|) \left(1+O\left(\log^{-1}(K|\mathcal{I}|)\right)\right) & n=2\ell,\\ \\
O\left(\log^\ell(K|\mathcal{I}|)\right) & n=2\ell+1. \end{array}\right. \]
More generally, let $0< h_1,\dots, h_n\leq (p-1)/2$. Then for any $K$ with $1/|\mathcal I| <K< d/n$,
\begin{eqnarray*}
\left< S^\pm(K, f, \psi^{h_1})\dots S^\pm(K, f, \psi^{h_n})  \right>_\FFd=
\begin{cases}
\frac{\Theta(h_1,\dots,h_n)}{(2\pi^2)^\ell}  \log^\ell(K|\mathcal{I}|) \left(1+O\left(\log^{-1}(K|\mathcal{I}|)\right)\right)
& n = 2 \ell, \\ \\
O\left(\log^\ell(K|\mathcal{I}|)\right)
& n = 2 \ell + 1. \end{cases} \end{eqnarray*}
The constant $\Theta(h_1,\dots,h_{n})$ is given by
\[\# \{(e_1,\dots,e_{n})\in \{-1,1\}, \sigma \in \mathbb{S}_{n}  \;:\; e_1h_{\sigma(1)}+e_2h_{\sigma(2)}\equiv \dots \equiv e_{2\ell-1}h_{\sigma(2\ell-1)}+e_{2\ell}h_{\sigma(2\ell)}\equiv 0 \, (\mathrm{mod}\, p)  \}\]
where $\mathbb{S}_{n}$ denotes the permutations of the set of $n$ elements.
\end{thm}

\begin{proof}
We have that
\[
\left< S^\pm(K, f, \psi^{h_1})\dots S^\pm(K, f, \psi^{h_n})  \right>_{\FF_d} =\sum_{{k_1, \dots, k_n =1}\atop{e_1, \dots, e_n = \pm 1}}^K
I_K^{\pm}(e_1 k_1) \dots I_K^{\pm}(e_n k_n) M_{n}^{{\bf k}, {\bf e}, {\bf h}},
\]
and we use Lemma \ref{generalcase2} to replace $M_{n}^{{\bf k}, {\bf e}, {\bf h}}$ in the sum.
The error term satisfies
\begin{eqnarray*}\sum_{{k_1, \dots, k_n =1}\atop{e_1, \dots, e_n = \pm 1}}^K
I_K^{\pm}(e_1 k_1) \dots I_K^{\pm}(e_n k_n) O \left( q^{3(k_1 + \dots +k_n)/2 - d/2} \right)
\ll \left( \sum_{k=1}^K q^{3k/2 - d/2n} \right)^n \ll 1
\end{eqnarray*}
when $d > 3nK$.

For the main term, we have to consider the sum of the terms
$T(k_1, \dots, k_n)$ from Lemma \ref{generalcase2}. For each fixed $T(k_1, \dots, k_n)$, we write the sum
over $k_1, \dots, k_n$
as $s$ nested sums $\Sigma_1 \dots \Sigma_s E_{\mathcal{F}_d}(m_1,\dots, m_s)$ where $\Sigma_w$ is a sum over the $k_i$ such that $i \in C_w$, and $| E_{\mathcal{F}_d}(m_1,\dots, m_s)| \ll 1$.
If $|C_w|=1$, then we have a sum
\begin{eqnarray} \sum_{k=1}^K \widehat{I}_K^{\pm}(k) q^{-k/2} \sum_{{m \mid k} \atop{m p \mid e_k}} \pi(m) m \ll 1,
\end{eqnarray}
because of Theorem \ref{boundforSpm}. For $r = |C_w|  \geq 2$, we have a sum of the type
\begin{eqnarray*} \sum_{k_1,\dots,k_r=1}^K\widehat{I}_K^\pm(e_1 k_1) \dots \widehat{I}_K^\pm(e_r k_r)q^{-(k_1+\dots+k_r)/2}\sum_{{m\mid (k_1,\dots,k_r)}\atop{mp \mid  \sum_{i=1}^r e_i h_ik_i}} \pi(m)m^r.
\end{eqnarray*}

When $r=|C_w|>2$, we will show in Lemma \ref{rbig} that the contribution from the terms of the sum over $k_1, \dots, k_r$ is
bounded. Assuming this result, we have by Lemma \ref{generalcase2}
that the leading term in  $S^{\pm}(K, f, \psi)^n$
will come from the contributions $N_{n, d}^{{\bf k}, {\bf e}, {\bf h}}$.

If $n=2\ell$, the leading terms are of the form
\begin{eqnarray*} &&\sum_{k_1,\dots,k_r=1}^K\widehat{I}_K^\pm(e_1 k_1) \dots \widehat{I}_K^\pm(e_r k_r)q^{-(k_1+\dots+k_r)/2} 
 \\
&&\times
\sum_{{m_1\mid (k_1,k_{\ell+1})}\atop{m_1p\mid e_1h_1 k_{\ell+1}+ e_{\ell+1}h_{\ell+1} k_{\ell+1}}} \pi(m_1) m_1^2 \dots \sum_{{m_\ell\mid (k_\ell,k_{2\ell})}\atop{m_\ell p\mid e_\ell h_\ell k_{2\ell}+ e_{2\ell}h_{2\ell} k_{2\ell}}} \pi(m_\ell) m_\ell^2 E_{\mathcal{F}_d}(m_1,\dots,m_\ell)
\end{eqnarray*}
By Definition \ref{definitionE} and Remark \ref{Eestimate} combined with  Proposition \ref{recordresult}, for $\mathcal{F}_d=\mathcal{F}_d^{\ord}, \mathcal{F}_d^{\full}$ the above sum gives
$$\left( \frac{1}{2\pi^2} \log{\left( K|\mathcal{I}| \right)} \right)^\ell.$$
For $\mathcal{F}^v_d$, we have that $E_{\mathcal{F}_d}(m_1,\dots,m_\ell)=1$ unless some of the $m_j$'s equal some of the $r_i$'s. Since the $r_i$'s are fixed constants, this simply introduces
an error term of the form $O\left(\log^{\ell-1}{\left( K|\mathcal{I}| \right)} \right)$ which does not change the final result.

If $n=2\ell+1$, the leading terms are of the
form
$$
O \left(\log^\ell{\left(K |\mathcal{I}|\right)} \right).
$$

The final coefficient is obtained by counting the numbers of ways to choose the $\ell$  coefficients $k_i$'s with positive sign ($e_i=1$) and to pair them with those with negative sign ($e_j=-1$).
\end{proof}

\begin{lem} \label{rbig} Let $r>2$, then
\[S:=\sum_{k_1,\dots,k_r=1}^K\widehat{I}_K^\pm(k_1) \dots \widehat{I}_K^\pm(k_r)q^{-(k_1+\dots+k_r)/2}\sum_{{m\mid (k_1,\dots,k_r)}\atop{mp\nmid (k_1,\dots,k_r)}} \pi(m)m^r=O(1)\]
\end{lem}
\begin{proof}
Suppose that $k_1\geq\dots\geq k_r$.  We use repeatedly the estimates from Remark \ref{Eestimate}. If $k_1=k_r$, we have
\[\sum_{{m\mid (k_1,\dots,k_r)}\atop{mp\nmid (k_1,\dots,k_r)}} \pi(m)m^r=O\left(k_1^{r-1} q^{k_1} \right).\]
If $k_1=2k_r$, and all the other $k_i$ are equal to $k_1$ or $k_r$, we have
\[\sum_{{m\mid (k_1,\dots,k_r)}\atop{mp\nmid (k_1,\dots,k_r)}} \pi(m)m^r=O\left(k_1^{r-1} q^{k_1/2} \right).\]
In all the other cases, the estimate is
\[\sum_{{m\mid (k_1,\dots,k_r)}\atop{mp\nmid (k_1,\dots,k_r)}}  \pi(m)m^r=O\left(k_1^{r-1} q^{k_1/3} \right).\]
Putting things together, we get
\begin{eqnarray*}S&\ll &\sum_{k=1}^K\widehat{I}_K^\pm(k)^r k^{r-1}q^{-(r-2)k/2}+\sum_{\ell=1}^{r-1}\sum_{k=1}^K\widehat{I}_K^\pm(2k)^\ell  \widehat{I}_K^\pm(k)^{r-\ell}k^{r-1}q^{(1-r/2-\ell/2)k}\\
&& +\sum_{k_1,\dots,k_r=1}^K\widehat{I}_K^\pm(k_1) \dots \widehat{I}_K^\pm(k_r)k_1^{r-1}q^{-k_1/6-(k_2+\dots+k_r)/2}\\&\ll& 1\end{eqnarray*}
by Proposition \ref{propmanysums}.
\end{proof}

\begin{rem}
 We note that if $n=2\ell$,
 \begin{equation}\label{combinatorics}
 \sum_{h_1, \dots, h_n=1}^{(p-1)/2} \Theta(h_1,\dots,h_n)=\frac{(p-1)^\ell (2 \ell)!}{2^\ell\ell! }.\end{equation}
There are $\frac{(2\ell)!}{\ell!2^\ell}$ ways of choosing unordered pairs of the form $\{e_i,e_j\}$. Inside each pair, exactly one of $\{e_i,e_j\}$ is positive and the other is negative, so there are a total $2^\ell$ choices for the signs. Finally, for each pair there are $(p-1)/2$ possible values for  $h_i$ which automatically determines the value of $h_j$.
\end{rem}

\begin{rem} By Theorem \ref{thm:Smoments}, the moments are given by sums of products of covariances. Thus, they are the same as the moments of a multivariate normal distribution. Moreover, the generating function of the moments converges due to \eqref{combinatorics}. Therefore, our random variables are jointly normal. Since the variables are uncorrelated (cf.~Theorem \ref{covariance}), it follows that our random variables (for $h=1, \dots, \frac{p-1}{2}$)  are independent.
\end{rem}

Recall that
\[
S^\pm(K,C_f)=\sum_{j=1}^{p-1}S^{\pm}(K, f, \psi^j).
\]

\begin{thm}\label{thm:sumisgaussian}
Assume that $K=\mathfrak{g}/\log\log(\mathfrak{g}|\mathcal I|)$, $\mathfrak{g} \rightarrow \infty$ and either  $|\mathcal{I}|$ is fixed or $|\mathcal{I}| \rightarrow 0$
while $\mathfrak{g}|\mathcal{I}| \rightarrow \infty$.
Then
\[
\frac{S^\pm(K, C_f)}{\sqrt{\frac {2(p-1)}{\pi^2}\log(\mathfrak{g}|\mathcal{I}|)}}
\]
 has a standard Gaussian limiting distribution when $\mathfrak{g} \rightarrow \infty$.
\end{thm}
\begin{proof}
First we compute the moments and then we normalize them.

With our choice of $K$ we have
\[
\frac{\log (K|\mathcal I|)}{ \log (\mathfrak{g}|\mathcal I |)}=1- \frac{\log\log\log(\mathfrak{g}|\mathcal I|)}{\log (\mathfrak{g}|\mathcal I |)} \rightarrow 1 \mbox{ as } \mathfrak{g}\rightarrow \infty.\]
Because of this, $\log (K|\mathcal I|)$ can be replaced by $\log (\mathfrak{g}|\mathcal I |)$ in our formulas.

Recall that $S^\pm(K, f, \psi^h)=S^\pm(K, f, \psi^{p-h})$, then
\begin{eqnarray*}
 S^\pm(K,C_f)^n=\left(2\sum_{h=1}^{(p-1)/2}S^\pm(K, f, \psi^h) \right)^n=2^n\sum_{h_1, \dots, h_n=1}^{(p-1)/2}S^\pm(K,f,\psi^{h_1})\dots S^\pm(K,f,\psi^{h_n}).
\end{eqnarray*}
Therefore, the moment is given by
\begin{eqnarray*}
\left\langle S^\pm(K, C_f)^n  \right\rangle_\FFd
&=&2^n\sum_{h_1, \dots, h_n=1}^{(p-1)/2}\langle S^\pm(K,f,\psi^{h_1})\dots S^\pm(K,f,\psi^{h_n})\rangle_\FFd.
\end{eqnarray*}
First assume that $n=2\ell$. By Theorem \ref{thm:Smoments}, this is asymptotic to
\begin{eqnarray*}
&&\frac{2^n}{(2\pi^2)^\ell}\log^\ell(\mathfrak{g}|\mathcal{I}|)\sum_{h_1, \dots, h_n=1}^{(p-1)/2} \Theta(h_1,\dots,h_n).
\end{eqnarray*}
Finally we use equation \eqref{combinatorics} to conclude that when $n=2\ell$,
\[
\left\langle S^\pm(K, C_f)^n  \right\rangle_\FFd\sim\frac{2^n(p-1)^\ell (2 \ell)!}{2^\ell \ell! (2\pi^2)^\ell}\log^\ell(\mathfrak{g}|\mathcal{I}|)=\frac{(2\ell)!}{\ell!\pi^{2\ell}} (p-1)^\ell\log^\ell(\mathfrak{g}|\mathcal{I}|).
\]
In particular, the variance is asymptotic to $\frac{2(p-1)}{\pi^2}\log(\mathfrak{g}|\mathcal{I}|)$.

Now assume that $n$ is odd, $n=2\ell+1$. Theorem \ref{thm:Smoments} yields
\begin{eqnarray*}
\left\langle S^\pm(K, C_f)^n  \right\rangle_\FFd
&=& O\left(\log^\ell (\mathfrak{g}|\mathcal{I}|) \right).
\end{eqnarray*}

Hence the normalized moment converges to
\[\lim_{\mathfrak{g}\rightarrow \infty} \frac{\left\langle S^\pm(K, C_f)^{2\ell} \right\rangle}{\left( \sqrt{\frac{2(p-1)}{\pi^2} \log (\mathfrak{g}|\mathcal I|)}\right)^{2\ell}}
= \frac{(2\ell)!}{\ell!2^\ell},\]
for $n=2\ell$, and to zero for $n$ odd. Hence, we have obtained the moments of the standard Gaussian distribution.
\end{proof}

\section{The distribution of zeroes}\label{proof}

We prove in this section that
$$ \frac{N_{\mathcal{I}}(C_f)
- 2\mathfrak{g} |\mathcal{I}|}{\sqrt{(2(p-1)/\pi^2) \log(\mathfrak{g} |\mathcal{I}|)}|}$$
converges in mean square to
$$\frac{S^{\pm}(K, C_f)}{{\sqrt{(2(p-1)/\pi^2) \log(\mathfrak{g} |\mathcal{I}|)}}}.$$
Then, using Theorem \ref{thm:sumisgaussian}, we get the result of Theorem \ref{zeroesthm}
since convergence in mean square implies convergence in distribution.

\begin{lem} Let $\mathcal{F}_d$ be any of the families under consideration. Assume that $K =\mathfrak{g}/\log \log(\mathfrak{g}|\mathcal{I}|)$, $\mathfrak{g}\rightarrow  \infty$ and either $|\mathcal{I}|$ is fixed or
$|\mathcal{I}|\rightarrow 0$ while $\mathfrak{g} |\mathcal{I}|\rightarrow \infty$. Then
\[\left< \left| \frac{N_\mathcal{I}(C_f) - 2\mathfrak{g} |\mathcal{I}| +S^{\pm}(K,C_f)}{\sqrt{(2(p-1)/\pi^2) \log (\mathfrak{g} |\mathcal{I}|)}}\right|^2 \right>_\FFd\rightarrow 0.\]
\end{lem}
\begin{proof}
From  equation \eqref{T-estimate-allzeroes-1}, using the Beurling--Selberg polynomials and the explicit formula (Lemma \ref{Explicit-Formula}),
we deduce that
\begin{eqnarray*}
 \frac{-2\mathfrak{g}}{K+1} \leq N_\mathcal{I}(C_f) - 2\mathfrak{g} |\mathcal{I}|  +S^-(K, C_f)  \leq S^-(K, C_f)
- S^+(K,C_f)   +\frac{2\mathfrak{g}}{K+1} \end{eqnarray*}
and
\begin{eqnarray*}
 \frac{-2\mathfrak{g}}{K+1} \leq -N_\mathcal{I}(C_f) + 2\mathfrak{g} |\mathcal{I}|  -S^+(K, C_f)   \leq S^-(K, C_f)
- S^+(K,C_f)   + \frac{2\mathfrak{g}}{K+1}.\ \end{eqnarray*}

Using these two inequalities to bound the absolute value of the central term, we obtain
\begin{eqnarray*}
&&\left< \left ( N_\mathcal{I}(C_f) - 2\mathfrak{g} |\mathcal{I}| + S^{\pm}(K, C_f) \right)^2 \right>_\FFd\\
&\leq& \max \left \{ \left(\frac{2\mathfrak{g}}{K+1} \right)^2, \left< \left(S^{-}(K, C_f) - S^{+}(K,C_f) +\frac{2\mathfrak{g}}{K+1}\right)^2 \right>_\FFd \right\} \\
&&\leq  \left(\frac{2\mathfrak{g}}{K+1} \right)^2 \\&+& \max \left \{ 0,  \left< \left( S^{-}(K, C_f) - S^{+}(K, C_f) \right)^2 \right>_\FFd
+ \frac{4\mathfrak{g}}{K+1} \left<  S^{-}(K, C_f) - S^{+}(K, C_f) \right>_\FFd \right\}.
\end{eqnarray*}

Now Theorem \ref{boundforSpm} implies that
\begin{eqnarray*}
 \left\langle S^{-}(K, C_f) - S^{+}(K, C_f) \right\rangle_\FFd &=& \left< S^{-}(K, C_f)  \right>_\FFd -
\left<   S^{+}(K, C_f) \right>_\FFd =O(1).
\end{eqnarray*}
For the remaining term we note that
\begin{eqnarray*}
&&\left\langle \left(S^{-}(K, C_f) - S^{+}(K, C_f)\right)^2\right \rangle_\FFd
\\
&=&\left\langle \left(S^{-}(K, C_f)\right)^2\right\rangle_\FFd  + \left\langle \left(S^{+}(K, C_f)\right)^2\right\rangle_\FFd -2\left\langle  \sum_{j_1, j_2=1}^{p-1} S^{-}(K, f, \psi^{j_1}) S^{+}(K, f, \psi^{j_2})\right\rangle_\FFd.
\end{eqnarray*}
By Corollary \ref{cor:2ndmoment}, this equals
\[\frac{4(p-1)}{\pi^2}\log(\mathfrak{g}|\mathcal I|)+O(1)-\frac{4(p-1)}{\pi^2}\log(\mathfrak{g}|\mathcal I|)+O(1)=O(1).\]
Therefore,
\[\left< \left ( N_\mathcal{I}(C_f )- 2\mathfrak{g} |\mathcal{I}| +S^{\pm}(K, C_f) \right)^2 \right>=O\left(\left(\frac{2\mathfrak{g}}{K+1}\right)^2\right)\]
and
\[\left \langle \left( \frac{N_\mathcal{I}(C_f) - 2\mathfrak{g} |\mathcal{I}|+S^\pm(K, C_f)}{\sqrt{(2(p-1)/\pi^2) \log(\mathfrak{g}|\mathcal{I}|)} }\right)^2 \right \rangle \rightarrow 0\]
when $\mathfrak{g}$ tends to infinity and $K =\mathfrak{g}/\log \log(\mathfrak{g}|\mathcal{I}|)$.
\end{proof}

\section{Acknowledgments} The authors would like to thank Rachel Pries for many useful discussions
while preparing this paper and Ze\'ev Rudnick for constructive comments  that helped clarify the exposition of the results. A substantial part of this work was completed during a SQuaREs program at the American Institute for Mathematics (AIM), and the authors would like to thank AIM for
this opportunity. The first and third named authors thank the Centre de Recherche Math\'ematique (CRM) for its hospitality. The fourth author thanks the Graduate Center at CUNY for its hospitality.

This work was supported by 
  the Simons Foundation [\#244988 to A. B.], the UCSD Hellman Fellows Program
[2012-2013 Hellman Fellowship to A. B.], the National Science Foundation [DMS-1201446 to B. F.], PSC-CUNY [to B.F.], the Natural Sciences and Engineering Research Council
 of Canada [Discovery Grant 155635-2008 to C. D., 355412-2008 to M. L.] and the Fonds de recherche du Qu\'ebec - Nature et technologies [144987 to M. L., 166534 to C. D. and M. L.]


\begin{thebibliography}{BDFL10b}

\bibitem[BDFL10a]{bdfl3}
Alina Bucur, Chantal David, Brooke Feigon, and Matilde Lal{\'{\i}}n.
\newblock Fluctuations in the number of points on smooth plane curves over
  finite fields.
\newblock {\em J. Number Theory}, 130(11):2528--2541, 2010.

\bibitem[BDFL10b]{bdfl1}
Alina Bucur, Chantal David, Brooke Feigon, and Matilde Lal{\'{\i}}n.
\newblock Statistics for traces of cyclic trigonal curves over finite fields.
\newblock {\em Int. Math. Res. Not. IMRN}, (5):932--967, 2010.

\bibitem[BDFL11]{bdfl2}
Alina Bucur, Chantal David, Brooke Feigon, and Matilde Lal{\'{\i}}n.
\newblock Biased statistics for traces of cyclic {$p$}-fold covers over finite fields.
\newblock In {\em WIN--Women in Numbers: Research Directions in Number Theory},
  volume 60 of {\em Fields Institute Communications Series}, pages 121--143.
  Amer. Math. Soc., Providence, RI, 2011.

\bibitem[BDFLS]{BDFLS}
Alina Bucur, Chantal David, Brooke Feigon, Matilde Lal{\'{\i}}n, and Kaneenika Sinha.
\newblock Distribution of zeta zeroes of Artin-Schreier covers.
\newblock  {\em Math. Res. Lett.}.
\newblock To appear, arXiv:1111.4701v2.

\bibitem[BK12]{bk}
Alina Bucur and Kiran S. Kedlaya.
\newblock The probability that a complete intersection is smooth.
\newblock {\em J. Th\'eor des Nombres de Bordeaux}, 24(3):541--556, 2012.


\bibitem [CWZ]{cwz}
GilYoung Cheong, Melanie Matchett--Wood, and Azeem Zaman.
\newblock The distribution of points on superelliptic curves over finite fields.
\newblock {\em Proc. Amer. Math. Soc.}
\newblock To appear,  arXiv:1210.0456.

\bibitem[Del74]{deligne}
Pierre Deligne.
\newblock La conjecture de {W}eil. {I}.
\newblock {\em Inst. Hautes \'Etudes Sci. Publ. Math.}, (43):273--307, 1974.

\bibitem[DS94]{ds}
Persi Diaconis and Mehrdad Shahshahani.
\newblock On the eigenvalues of random matrices.
\newblock Studies in applied probability. {\em J. Appl. Probab.}, 31A:49--62, 1994.

\bibitem[Ent12]{entin}
Alexei Entin.
\newblock On the distribution of zeroes of {A}rtin-{S}chreier {$L$}-functions.
\newblock {\em Geom. Funct. Anal.}, 22(5):1322-1360, 2012.

\bibitem[EW]{ew}
Daniel Erman, Melanie Matchett--Wood.
\newblock Semiample Bertini theorems over finite fields.
\newblock  Preprint, arXiv:1209.5266v1.

\bibitem[FR10]{fr}
Dmitry Faifman and Ze{\'e}v Rudnick.
\newblock Statistics of the zeros of zeta functions in families of
  hyperelliptic curves over a finite field.
\newblock {\em Compos. Math.}, 146(1):81--101, 2010.

\bibitem[Gar05]{garcia}
Arnaldo Garcia.
\newblock On curves over finite fields.
\newblock {\em Arithmetic, geometry and coding theory (AGCT 2003)} S\'emin. Congr., 11, Soc. Math. France: 75--110, 2005.

\bibitem[Kat87]{Katz1}
Nicholas M. Katz.
\newblock On the monodromy groups attached to certain families of exponential
  sums.
\newblock {\em Duke Math. J.}, 54(1):41--56, 1987.


\bibitem[KS99]{KS} Nicholas M. Katz and Peter Sarnak.
\newblock {\em Random matrices, Frobenius eigenvalues, and monodromy}.
 \newblock American Mathematical Society Colloquium Publications, 45. American Mathematical Society, Providence, RI, 1999. xii+419 pp.

\bibitem[KR09]{kr}
P{\"a}r Kurlberg and Ze{\'e}v Rudnick.
\newblock The fluctuations in the number of points on a hyperelliptic curve
  over a finite field.
\newblock {\em J. Number Theory}, 129(3):580--587, 2009.



\bibitem[Mon94]{M}
Hugh L. Montgomery.
\newblock {\em Ten lectures on the interface between analytic number theory and
  harmonic analysis}, volume 84 of {\em CBMS Regional Conference Series in
  Mathematics}.
\newblock Published for the Conference Board of the Mathematical Sciences,
  Washington, DC, 1994.

\bibitem[PZ12]{pz}
Rachel Pries and Hui~June Zhu.
\newblock The {$p$}-rank stratification of {A}rtin-{S}chreier curves.
\newblock {\em Ann. Inst. Fourier (Grenoble)}, 62(2):707--726, 2012.


\bibitem[Ros02]{rosen}
Michael Rosen.
\newblock {\em Number theory in function fields}, volume 210 of {\em Graduate
  Texts in Mathematics}.
\newblock Springer-Verlag, New York, 2002.

\bibitem[GV92]{gv}
Gerard van der Geer and  Marcel van der Vlugt.
\newblock{Reed-Muller codes and supersingular curves. I.}
\newblock {\em Compositio Math.} 84(3):333-367, 1992.

\bibitem[Woo12]{wood}
Melanie Matchett--Wood.
\newblock The distribution of the number of points on trigonal curves over
  {$\mathbb F_q$}.
\newblock {\em Int. Math. Res. Not. IMRN}, (23):5444--5456, 2012.

\bibitem[Xio10a]{x2}
Maosheng Xiong.
\newblock The fluctuations in the number of points on a family of curves over a
  finite field.
\newblock {\em J. Th\'eor. Nombres Bordeaux}, 22(3):755--769, 2010.

\bibitem[Xio10b]{x1}
Maosheng Xiong.
\newblock Statistics of the zeros of zeta functions in a family of curves over
  a finite field.
\newblock {\em Int. Math. Res. Not. IMRN}, (18):3489--3518, 2010.


\bibitem[Xio]{x3}
Maosheng Xiong.
\newblock{Distribution of zeta zeroes for abelian covers of algebraic curves over a finite field.}
\newblock Preprint, arXiv:1301.7124.


\bibitem[Zhu]{zhu}
Hui June Zhu.
\newblock{Some families of supersingular Artin-Schreier curves in characteristic $>2.$}
\newblock Preprint, arXiv:0809.0104.



\end{thebibliography}
\end{document}